\documentclass{article}
\usepackage{graphicx} 
\usepackage[utf8]{inputenc}
\usepackage{amsmath,amsthm,amsfonts,amssymb,mathbbol}
\usepackage{slashed}
\usepackage{appendix}
\usepackage{tikz}
\usepackage{epsfig}
\usepackage{mathrsfs}
\usepackage{amscd} 
\usepackage[all]{xy}

\usepackage{appendix}

\usepackage{dsfont, upgreek}
\usepackage{mathtools}
\usepackage{epsfig}

\usepackage[all]{xy}
\usepackage{tikz}
\usepackage{tikz-cd}
\usetikzlibrary{matrix,arrows,decorations.pathmorphing}
\usetikzlibrary{topaths}

\theoremstyle{plain}
\newtheorem{theorem}{Theorem}[section]
\newtheorem{proposition}[theorem]{Proposition}
\newtheorem{lemma}[theorem]{Lemma}
\newtheorem{corollary}[theorem]{Corollary}

\theoremstyle{definition}
\newtheorem{definition}[theorem]{Definition}

\theoremstyle{remark}
\newtheorem{remark}[theorem]{Remark}

\newcommand{\D}{\slashed{D}}

\begin{document}

\title{Equivariant Localization of $K$-Homological Euler Class for Almost Connected Lie Groups}

\date{\today}
\author{Hongzhi Liu \footnote{School of Mathematics, Shanghai University of Finance and Economics, 777 Guoding Rd, Shanghai 200433, P. R. China.  \newline Email: liu.hongzhi@sufe.edu.cn.  This is the corresponding author.}, Hang Wang  \footnote{Research Center for Operator Algebras, School of Mathematical Sciences 
East China Normal University, 3663 North Zhongshan Rd, 
Shanghai 200062, P.R.China \newline Email: wanghang@math.ecnu.edu.cn}, 
Zijing Wang \footnote{Shanghai Institute for Mathematics and Interdisciplinary Sciences (SIMIS), Research Institute of Intelligent Complex Systems, Fudan University, Shanghai {\rm 200433}, China. \newline Email: wzj@simis.cn}, 
Shaocong Xiang \footnote{School of Mathematics, Shanghai University of Finance and Economics, 777 Guoding Rd, Shanghai 200433, P. R. China. \newline Email: xiangshaocong@mail.sufe.edu.cn} }

\maketitle

\begin{abstract}
 
Using the Witten deformation and localization algebra techniques, we compute the $G$-equivariant $K$-homology class of the de Rham operator on a proper cocompact $G$-spin manifold, where $G$ is an almost connected Lie group. By applying a $G$-invariant Morse–Bott perturbation, this class is localized  near the zero set of the perturbation and can be identified explicitly with an element in the representation rings associated to some isotropy subgroups. The result yields an equivariant Poincar\'e–Hopf formula and supplies concise tools for equivariant index computations.

\end{abstract}

\section{Introduction}

The Witten deformation \cite{witten1982supersymmetry} and the analytic localization framework have long been employed to compute indices of de Rham operators (see \cite{mb1,bismut1991complex,z}). We encode these foundational ideas in the formulation of an equivariant $K$-homology cycle represented by the de Rham operator, in the context of Yu’s localization algebra \cite{yu20book,yu1997localization}, extended to an equivariant setting with respect to an almost connected Lie group.
This extension provides a flexible analytic framework for studying equivariant index theory on noncompact manifolds equipped with proper actions.

Localization algebras serve as powerful models for $K$-homology. They have been extensively applied to the study of the Baum-Connes conjecture  and the Novikov conjecture \cite{Yu98asymptotic, Yu00coarseemdedding, KasparovYu06, Yu10characterization, ChenWangYu13, ChenWangYu15, FuWangYu20, GongWuYu21, DengWangYu23, DengGuoWangYu25}, as well as to the definition and analysis of the higher rho invariant \cite{XieYu14, Zeidler16, XieYu17rhoandmoduli, WeibergerXieYu21, ChenLiuYu20, XieYuZeidler21, GuoXieYu21, LiuWang21, XieYu21delocalized, ChenLiuWangYu22, ChenWangXieyu23, WangXieYu23}, among other geometric and topological contexts. Our results furnish a new example illustrating the versatility of the localization algebra.

Our motivation for studying the $K$-homological formulation of the Witten deformation traces back to the work of Rosenberg \cite{rosenberg99}, where he employed the spirit of Witten deformation to show that the $K$-homology class of the de Rham operator is, in a sense, as trivial as possible—namely, it arises from the one-point embedding. This perspective was later refined in the equivariant setting by Lück and Rosenberg \cite{luckroseberg03}, who showed that for a countable discrete group $G$ and a proper cocompact $G$-manifold $M$ without boundary, equipped with a $G$-invariant Riemannian metric, the $G$-equivariant $K$-homology class of the de Rham operator is determined by the so-called universal equivariant Euler characteristic, which can be described in terms of representation rings of some isotropy subgroups.

In \cite{liu2024equivariant}, we provided a new proof of the main results of L\"uck and Rosenberg by employing Yu’s localization algebra framework \cite{yu1997localization} and the pairing between $K$-homology and $K$-theory in this context \cite[Chapter 9]{yu20book}. More precisely, let $G$ be a countable discrete group acting properly and cocompactly on a complete even-dimensional manifold $M$. The de Rham operator $D$ on $M$ then defines a class $[D]$ in the $K$-theory of the $G$-equivariant localization algebra $C^*_L(M)^G$, which is canonically isomorphic to the $G$-equivariant $K$-homology of $M$.
Let $A$ be a $G$-invariant cotangent vector field on $M$ with only nondegenerate zeros. We observed that, due to the uniform invertibility of the Witten-type deformation $\frac{D}{t} + A$ for $t \in [1, \infty)$, the class $[D] = [D + A]$ can be localized to neighborhoods of the zero set of $A$. This localization reduces the computation of $[D]$ to the pairing between $K$-homology and $K$-theory in Euclidean space, specifically that of $\mathbb{R}^n$.

In this paper, we generalize this approach to the setting where \(G\) is an almost connected Lie group acting properly and cocompactly on a manifold without boundary. A new phenomenon appears in this non-discrete setting. Namely, the zero set of the equivariant vector field associated with a \(G\)-invariant Morse--Bott \(1\)-form is in general no longer a discrete set of points; even in favorable situations, it consists of non-trivial zero submanifolds. As already indicated in Bismut's work~\cite{mb1}, such zero submanifolds create substantial difficulties for the localization of the de~Rham operator.

The key point is that, after localization, the contribution of each
zero component is governed not merely by the geometry of the component
itself, but more fundamentally by the transverse geometry encoded in
its normal bundle. To retain this normal information, we introduce the
\(G\)-equivariant vertically representable localization algebra
\(RVC_L^*(E)^G\), associated to a \(G\)-equivariant vector bundle \(E\)
over a zero component; see Definition~\ref{def rvcl}. This algebra
provides a natural framework in which the localized de~Rham class can
be represented on the total spaces of the normal bundles of the zero
components.

Our main theorem describes the localization of the
\(G\)-equivariant de~Rham class and the structure of the resulting
local contributions. It first shows that the \(G\)-equivariant
\(K\)-homology class of the de~Rham operator localizes to the normal
bundles of the zero components. Through equivariant Poincar\'e duality
and induction, the contribution associated to \(Z_i=G/H_i\) is then
represented by a class
\[
\kappa_i\in K_{H_i}^0(V_i),
\]
where \(V_i\) is the normal fiber over \(eH_i\). Moreover, when the
normal bundle of \(Z_i\) is \(G\)-orientable, the class \(\kappa_i\)
admits an explicit description in terms of the local degree and the
half-spin representations associated to the stabilizer action on
\(V_i\).

\begin{theorem}\label{thm main}
Let \(G\) be an almost connected Lie group acting properly,
cocompactly, and isometrically on a \(G\)-spin manifold \(X\).
Suppose that \(X\) admits a \(G\)-invariant Morse--Bott \(1\)-form
\(A\) whose zero set decomposes as
\[
Z=\bigsqcup_{i=1}^k Z_i
=\bigsqcup_{i=1}^k G/H_i,
\]
where each \(H_i\leq G\) is compact. For each \(i\), let
\(E_i\to Z_i\) be the normal bundle of \(Z_i\) in \(X\).

\textup{(1)}
The \(G\)-equivariant \(K\)-homology class \([d+d^*]\) lies in the
image of the homomorphism
\begin{equation}\label{eq:localization-map}
\iota_*:
\bigoplus_{i=1}^k K_0\!\bigl(RVC_L^*(E_i)^G\bigr)
\longrightarrow
K_0\!\bigl(C_L^*(X)^G\bigr),
\end{equation}
induced by \(G\)-equivariant tubular neighborhood embeddings
\(E_i\hookrightarrow X\).

\textup{(2)}
Let \(V_i=(E_i)_{eH_i}\) be the normal fiber over \(eH_i\in G/H_i\). 
The \(G\)-equivariant Poincar\'e duality isomorphism and the induction 
isomorphism yield
\[
K_0\!\bigl(RVC_L^*(E_i)^G\bigr)
\cong K_G^0(E_i)
\cong K_{H_i}^0(V_i).
\]
Under these identifications, the homomorphism
\eqref{eq:localization-map} corresponds to
\begin{equation}\label{eq:localization-map-2}
T:
\bigoplus_{i=1}^k K_{H_i}^0(V_i)
\longrightarrow
K_0\!\bigl(C_L^*(X)^G\bigr).
\end{equation}

For each \(i\), let \(\tau_i\) denote the \(H_i\)-equivariant twisting
determined by the Clifford algebra \(\mathrm{Cl}(V_i)\). The equivariant
twisted Thom isomorphism gives a natural identification
\[
K_{H_i}^0(V_i)
\cong K_{H_i}^{\tau_i}(\mathrm{pt})
\cong R^{\tau_i}(H_i),
\]
where \(R^{\tau_i}(H_i)\) is the \(\tau_i\)-twisted representation group of \(H_i\). There exist classes
\[
\kappa_i\in R^{\tau_i}(H_i),\qquad 1\leq i\leq k,
\]
such that, upon identifying them with their corresponding classes in
\(K_{H_i}^0(V_i)\),
\[
T\left(\bigoplus_{i=1}^k\kappa_i\right)=[d+d^*].
\]

\textup{(3)}
If, for some \(i\), the zero submanifold \(Z_i\) is
\(G\)-orientable, then the corresponding class \(\kappa_i\) can be chosen to be
\[
\kappa_i=
\begin{cases}
0,
& \dim V_i \text{ is odd},\\[4pt]
\deg\bigl(A|_{Z_i}\bigr)
\bigl([s_i^+]-[s_i^-]\bigr),
& \dim V_i \text{ is even},
\end{cases}
\]
where \(s_i^\pm\) are the half-spin representations of the double cover
\(\widetilde H_i\) of \(H_i\) defined in Definition~\ref{def:pin-double-cover}, and
\(\deg\bigl(A|_{Z_i}\bigr)\) is defined in Proposition~\ref{sg}.
\end{theorem}

\begin{remark}
The classes \(\kappa_i\) in Theorem~\ref{thm main} are not uniquely
determined in general. Indeed, the homomorphism
\[
T:
\bigoplus_{i=1}^k K_{H_i}^0(V_i)
\longrightarrow
K_0\!\bigl(C_L^*(X)^G\bigr)
\]
need not be injective, so different collections of local classes may have
the same image under \(T\). Thus, the significance of
Theorem~\ref{thm main}\textup{(3)} is not that it determines a canonical
choice of \(\kappa_i\), but rather that it provides one particular choice
for which the local contribution admits an explicit
representation-theoretic description.

If the \(H_i\)-representation \(V_i\) is even-dimensional and admits an
\(H_i\)-equivariant spin structure, the corresponding Clifford twist is
trivialized, and the equivariant Thom isomorphism identifies
\(K_{H_i}^0(V_i)\) with the ordinary representation ring \(R(H_i)\).
Consequently, the local class may be expressed in terms of ordinary
\(H_i\)-representations. In general, however, the Clifford algebra
\(\mathrm{Cl}(V_i)\) determines a possibly nontrivial equivariant twist
\(\tau_i\), and the Thom isomorphism identifies
\[
K_{H_i}^0(V_i)
\cong K_{H_i}^{\tau_i}(\mathrm{pt})
\cong R^{\tau_i}(H_i).
\]
Here \(R^{\tau_i}(H_i)\) denotes the Grothendieck group of
\(\tau_i\)-twisted \(H_i\)-representations, which may equivalently be
described in terms of suitably graded \(\mathrm{Cl}(V_i)\)-modules
equipped with a compatible \(H_i\)-action. Thus, even when \(V_i\) does
not admit an equivariant spin structure, the local contribution still
has a natural representation-theoretic interpretation, but in a
twisted rather than an ordinary representation ring.

It is also useful to compare Theorem~\ref{thm main} with the discrete-group case studied in~\cite{liu2024equivariant}. There, the zero set consists of isolated points, and the local contribution associated to a zero orbit is given by the local index
\[
\operatorname{ind}(\Xi,x_i)\in\{\pm1\},
\]
of a vector field \(\Xi\) multiplied by the class of the trivial representation of the stabilizer \(H_i\). In the present Morse--Bott setting, the two factors
\[
\deg\bigl(A|_{Z_i}\bigr)
\qquad\text{and}\qquad
[s_i^+]-[s_i^-]
\]
should be regarded as the corresponding analogues for a
positive-dimensional zero submanifold. The first generalizes the local
index at an isolated zero by recording the transverse degree of \(A\)
along \(Z_i\), whereas the second records the action of the stabilizer on the normal directions and replaces the trivial stabilizer representation appearing in the isolated-zero case. Accordingly,
Theorem~\ref{thm main} may be viewed as a refinement of the
discrete-group localization formula, in which the local sign and the
trivial representation are replaced by their natural
normal-directional counterparts.
\end{remark}

The proof of Theorem~\ref{thm main} builds on the localization-algebraic
formulation of Witten deformation in equivariant \(K\)-homology developed
in~\cite{liu2024equivariant}, extending that framework from discrete group
actions to the present almost connected Lie group setting. The other main
ingredients are equivariant Poincar\'e duality, induction, and a local
analysis of the normal representations along the zero orbits.

First, we introduce the \(G\)-equivariant localization algebra
\(C_L^*(X)^G\) in Definition~\ref{def cl}. We show that, for proper
cocompact actions, there is an isomorphism
\[
K_0\bigl(C_L^*(X)^G\bigr)\cong K_0^G(X);
\]
see Theorem~\ref{thm qualify our k homo}. This identifies the class of
the de~Rham operator with an element
\[
[D]\in K_0\bigl(C_L^*(X)^G\bigr).
\]

We then choose a \(G\)-invariant Morse--Bott \(1\)-form \(A\) on \(X\);
see~\cite{equivariantdifferentialtopology} and
Definition~\ref{def morsebott covector filed}. Its zero components are
assumed to be nondegenerate \(G\)-orbits. Using the uniform invertibility
of the Witten perturbation
\[
\frac{D}{t}+A,\qquad t\in[1,\infty),
\]
away from the zero set, we localize the class \([D]\) to tubular
neighborhoods of the zero components. Equivalently, \([D]\) lies in the
image of the homomorphism
\[
\iota_*:
\bigoplus_i K_0\!\bigl(RVC_L^*(E_i)^G\bigr)
\longrightarrow
K_0\!\bigl(C_L^*(X)^G\bigr),
\]
where \(E_i\) is the normal bundle of \(Z_i\) in \(X\); see
Definition~\ref{def rvcl} and Theorem~\ref{thm the cut off}. This proves
the first part of Theorem~\ref{thm main}.

We next establish Poincar\'e duality in the vertically representable
setting.

\begin{theorem}\label{thm pd for bundle in intro}
Let \(M\) be a Riemannian manifold equipped with a proper, cocompact,
and isometric action of \(G\), and let \(E\to M\) be a
\(G\)-equivariant vector bundle. Assume that the total space of \(E\)
is a \(G\)-spin manifold. Then the homomorphisms
\begin{align*}
\otimes [\D^E]&:
K_G^0(E)\longrightarrow K_0\!\bigl(RVC_L^*(E)^G\bigr),\\
\otimes [\theta_E]&:
K_0\!\bigl(RVC_L^*(E)^G\bigr)\longrightarrow K_G^0(E),
\end{align*}
defined in Subsection~\ref{subsec thom and poincare}, are inverse to
each other.
\end{theorem}

A principal difficulty in the proof of this theorem is that the
cutting-and-pasting argument, which is a central tool in the study of
equivariant localization algebras for discrete groups, does not apply
directly to continuous group actions.

Applying Theorem~\ref{thm pd for bundle in intro} to the normal bundles
identifies the localized classes with elements of
\[
\bigoplus_i K_G^0(E_i).
\]
For a zero orbit \(Z_i=G/H_i\), the normal bundle is of the form
\[
E_i\cong G\times_{H_i}V_i,
\qquad
V_i=(E_i)_{eH_i}.
\]
Equivariant induction therefore gives a natural isomorphism
\[
K_G^0(E_i)\cong K_{H_i}^0(V_i).
\]
Composing these identifications with the embedding-induced map
\(\iota_*\), we obtain the homomorphism
\[
T:
\bigoplus_i K_{H_i}^0(V_i)
\longrightarrow
K_0\!\bigl(C_L^*(X)^G\bigr)
\]
appearing in Theorem~\ref{thm main}. The localized de~Rham class is
thus represented by a collection of classes
\[
\kappa_i\in K_{H_i}^0(V_i).
\]
This proves the second part of the theorem. When the normal bundle is
\(G\)-orientable, a further analysis of the normal operator gives the
explicit representative stated in Theorem~\ref{thm main}\textup{(3)}.
In the general case, the groups \(K_{H_i}^0(V_i)\) may also be expressed
in twisted representation-theoretic terms; this interpretation will be
developed later and is not needed for the localization statement itself.

This explicit identification of the \(K\)-homological Euler class under the map \(T\) opens up several directions for future research. In particular, it provides a potential bridge between index theory and representation-theoretic or topological invariants of almost connected Lie groups. We expect that this framework can be further developed to study the existence of nontrivial \(L^2\)-harmonic forms, and to explore connections with invariants such as \(L^2\)-cohomology, \(L^2\)-Betti numbers, and their delocalized analogues.

The paper is organized as follows. In Section \ref{sec preliminary}, we introduce some fundamental concepts, including the almost connected Lie group equivariant localization algebra, the almost connected Lie group equivariant vertically representable localization algebra, the definition of the equivariant  $K$-homology class of a Dirac-type operator, and provide a brief review of the pairing between equivariant $K$-homology and equivariant $K$-theory, as investigated in \cite[Chapter 9]{yu20book}. We shall also prove Theorem \ref{thm pd for bundle in intro} to establish the Poincar\'e duality between the equivariant $K$-homology of the equivariant $K$-theory. In Section \ref{sec main proof} and Section \ref{subsec preimage computation}, we show that the equivariant $K$-homology class of the de Rham operator can be localized via the Witten deformation, and prove the main theorem, Theorem~\ref{thm main}.

\section*{Acknowledgments} 
HL is supported by NSFC 11901374. HW is supported by grants 23JC1401900 and NSFC 12271165, and in part by the Science and Technology Commission of Shanghai Municipality under grant No. 22DZ2229014.
The authors are grateful to Guoliang Yu for valuable discussions.

\section{Almost connected Lie group equivariant Localization algebra and $K$-homology}\label{sec preliminary}

In this section, we introduce fundamental concepts related to the almost connected Lie group equivariant localization algebra, the equivariant $K$-homology, and the construction of the equivariant $K$-homology classes for Dirac-type operators. Furthermore, we revisit the construction of the pairing  between $K$-homology and $K$-theory in the equivariant setting, as developed by Willett and Yu in \cite{yu20book}.
Throughout this paper, we assume that all group actions are proper, isometric and cocompact.  
\subsection{Equivariant localization algebra}\label{subsec localization algebra}

In this subsection, we introduce the notion of almost connected Lie group equivariant localization algebra. To provide context, we first recall the almost connected Lie group equivariant Roe algebra, as defined and studied in \cite{guo2021equivariant}.

Throughout this subsection, let $X$ be a complete manifold equipped with an isometric, proper and cocompact  action by an almost connected Lie group $G.$

Let $E$ be a $G$-equivariant bundle over $X.$ Set $H(E)$ as the $G-C_0(X)$ module $L^2(E)\otimes L^2(G),$ where $G$ acts on $H(E)$ diagonally, and $C_0(X)$ acts on $H(E)$ by the pointwise multiplication on $L^2(E).$

We now define the concepts of propagation and local compactness for a bounded operator on an ample $G$-equivariant $X$-module.
\begin{definition}
    Let $E$ be a $G$-equivariant bundle over $X.$  Let $T$ be an operator in $B(H(E)).$ 
    \begin{enumerate}
     \item The support of $T,$ denoted by $\text{supp}(T),$
 is defined as the complement of the set
 \[\{(x,y)\in X\times X|\exists f, g\in C_0(X), s.t. f(x)\neq 0, g(y)\neq 0, fTg=0\}.\]
 \item 
 The operator $T$ is said to have finite propagation if
			\[
			\text{sup}\{d(x,y): (x,y)\in \text{supp}(T)\}< \infty
			\]
			where this supremum is referred to as the propagation of $T$, denoted by $\text{propagation}(T)$.
            \item An operator $B$ is said to be $G$-continuous if both $g\mapsto gB$ and $g\mapsto Bg$ are continuous with respect to $g\in G.$ 
\item The operator $T$ is said to be locally $G$-continuous compact if $fT$ and $Tf$ are $G$-continuous compact operators for all $f\in C_0(X)$.
 \end{enumerate}
     
\end{definition}

\begin{definition}\label{def cl}
   Assume that $G$ acts on $X$ properly, isometrically. Let $E$ be a $G$-equivariant bundle over $X.$ 
   \begin{enumerate}
       \item The algebraic $G$-equivariant Roe algebra for $H(E)$ is the algebra $C^*_{alg}(X; H(E))^G,$ consisting of the  locally $G$-continuous compact operators in $B(H(E))^G$ with finite propagation. The equivariant Roe algebra for $H(E)$ is the closure $C^*(X; H(E))^G$ of $C^*_{alg}(X; H(E))^G$ in $B(H(E)).$ This definition is essentially from \cite{guo2021equivariant}.


    \item The $G$-equivariant localization algebra, denoted as 
   \[C_L^*(X, H(E) )^G\] 
   is defined as the $C^*$-algebra generated by all bounded and uniformly norm-continuous functions $f:[1,\infty) \to C^*(X, H(E) )^G$ such that the propagation of $f(t)$ tends to zero as $t\to \infty.$ 
   \end{enumerate}
   
\end{definition} 

We now explain why we restrict ourselves to modules of the form
\[H(E)=L^2(E)\otimes L^2(G) .\]
First,  modules of this form arise naturally in geometric contexts and are admissible in the sense of  \cite{guo2021equivariant}. This implies that the image of the evaluation map
\begin{eqnarray*}
    \text{ev}: C_L^*(X,H(E))^G&\to &C^*(X, H(E) )^G\\
    f&\mapsto & f(1) 
\end{eqnarray*}
lies in $C^*_r(G)\otimes \mathcal{K}(\mathcal{H}),$ given that the action is cocompact. 

Secondly, one can show that the definition of  $C_L^*(X, H(E))^G$ does not depend on the choice of the $G$-equivariant Riemannian bundle $E$ at the $K$-theory level.   To illustrate this, let $X$ and $X'$ be two proper $G$-manifolds, $E$ and $E'$ are $G$-equivariant Riemannian bundles over $X$ and $X'$ respectively. Let $f: X\to X'$ be a $G$-equivariant continuous coarse map (for definition, see \cite[Section 2]{yu1997localization}). 
\begin{lemma}\label{lem covering map}
    For any $\epsilon>0,$ there exists a $G$-equivariant isometry $V_f: L^2(E)\otimes L^2(G)\to  L^2(E')\otimes L^2(G),$ such that the support of $V_f$  
    \[\text{supp}(V_f)\subset \{ (x,x')\in X\times X': d(f(x),x')\leq \epsilon  \}.\]
    Recall that the support of $V_f$ refers to the complement of the set of all points $(x,x')\in X\times X'$ for which there exist functions $g\in C_0(X), g'\in C_0(X')$ such that $gV_f g'=0,$ and $g(x)\neq 0, \ g'(x')\neq 0.$
\end{lemma}
\begin{proof}
    For any subset $U\subset X',$ set $U^G$ as 
    \[
    \{(gx,g)\in X'\times G: x\in U, g\in G\}
    \]
    and for any $g\in G,$ define  
    \[U^g:=\{(gx,g)\in X'\times G: x\in U\}.\]
    By omitting a zero measure subset, $X'$ can be decomposed as the disjoint union of open subsets
    \[\bigsqcup_i U_i,\]
    such that the diameter of $U_i$ is less than $\epsilon.$
    With respect to this disjoint union, one can see that 
    \[
    L^2(E')\otimes L^2(G)\cong L^2(E'\times G)=\bigoplus_i L^2(U^G_i).
    \]
    Since both $L^2(f^{-1}(U_i^g))$ and $L^2(U_i^g)$ are  infinite dimensional, one can construct an isometry from $L^2(f^{-1}(U_i^g))$ to $L^2(U_i^g),$ which naturally gives rise to a $G$-equivariant isometry $ v_i :L^2(f^{-1}(U_i^G))\to L^2(U_i^G).$ 
    Now we define $V_f$ as $\bigoplus_i v_i.$ By definition, one can see that the support of $V_f$ lies in the set 
      \[\{ (x,x')\in X\times X': d(f(x),x')\leq \epsilon  \}.\]
      The proof is then completed.
\end{proof} 
By the above lemma, one can see that for each $k\in \mathbb{N}_+$ there exists $V_f(k)$ such that
\[\text{supp}V_{f}(k)\subset \{ (x,x')\in X\times X': d(f(x),x') \leq \frac{1}{k}\}.\]
Now we define a family of isometries $V_f(t): H(E)\oplus H(E)\to H(E')\oplus H(E')$  by
\[
V_f(t)=R(t-k) \begin{pmatrix}
    V_f(k) &0 \\
    0 &  V_f(k+1) 
\end{pmatrix} R^{-1}(t-k), \quad \forall k\leq t\leq k+1,
\]
where 
\[
R(t)=\begin{pmatrix}
    \cos\frac{\pi}{2}t&  \sin\frac{\pi}{2}t \\
     -\sin\frac{\pi}{2}t  &  \cos\frac{\pi}{2}t 
\end{pmatrix}, \quad t\in [0,1].
\]
This family of isometries $V_f(t)$ induces a homomorphism 
\[Ad(V_f (t)):(C_L^*(X, H(E))^G)^+\to (C_L^*(X', H(E'))^G)^+ \otimes M_2(\mathbb{C})\] defined as 
\[Ad(V_f(t)) (u(t)+cI)= V_f(t) (u(t)\oplus 0) V_f^*(t) +cI, \forall u(t)\in C_L^*(X, H(E))^G, c\in \mathbb{C}. \]
Note that $ V_f(t) (u(t)\oplus 0) V_f^*(t)$ is continuous although $V_f(t)$ is not.

\begin{lemma}\label{lem of adv}
   Let $f,$ $X,\ X',\ E, \ E',$ $Ad(V_f(t))$ be as above. Then $Ad(V_f(t))$ induces a homomorphism from $(C_L^*(X, H(E))^G)^+$ to $ (C_L^*(X', H(E'))^G)^+ \otimes M_2(\mathbb{C}),$ which is independent of the choice of $V_f$ at the $K$-theory level. As a result, the definition of $C_L^*(X, H(E))^G$ does not depend on the choice of the $G$-equivariant Riemannian bundle $E.$ 
\end{lemma}

Thus, in the following, we denote the $G$-equivariant  localization algebra of $X$ as $C^*_L(X)^G.$

Now let $E\to X$ be a $G$-equivariant Riemannian bundle of rank $n$. We define the $G$-equivariant vertically representable localization algebra $RVC^*_L(E)^G$. The definition below is inspired by \cite[Section 9.4]{yu20book}.

\begin{definition}\label{def rvcl}
Let $\mathcal{E}$ be a $G$-equivariant bundle over $E$. We define the vertically representable localization algebra
\[
RVC^*_L(E, L^2(\mathcal{E})\otimes H)^G
\]
to be the $C^*$-subalgebra of
\[
C_L^*(E,L^2(\mathcal{E})\otimes H)^G
\]
generated by all elements $t\mapsto T(t)$ for which there exist a vertically compact subset $K$ and $t_K\geq 1$ such that
\[
T(t)=\chi_K T(t) \chi_K,\quad \text{for all } t\geq t_K.
\]
For the same reason as explained after Definition \ref{def cl}, the definition of the vertically representable localization algebra is independent of the choice of $\mathcal{E}$ at the level of $K$-theory. Thus, in the following, we abbreviate it as $RVC^*_L(E)^G$.
\end{definition}

We remark that Definition \ref{def rvcl} applies equally well to tubular neighborhoods.

Now we are ready to define the $G$-equivariant $K$-homology and the $G$-equivariant vertically representable $K$-homology.
\begin{definition}\label{def K homo}
    Let $X$ be a complete manifold with a proper, cocompact and isometric $G$-action, where $G$ is an almost connected Lie group. The $G$-equivariant $K$-homology of $X$ is defined as the $K$-theory of $C^*_L(X)^G.$

    Let $E\to X$ be a $G$-equivariant  Riemannian $\mathbb{R}^n$ bundle over $X.$ The vertically representable $G$-equivariant $K$-homology of $E\to X$ is defined as the $K$-theory of $RVC^*_L(E)^G.$

\end{definition}

The following definition is the $G$-equivariant version of \cite[Definition 3.5]{yu1997localization}. 
\begin{definition}\label{def strong lip}
    Let $X$ and $X'$ be two complete manifolds with cocompact, proper and isometric  $G$-actions, where $G$ is an almost connected Lie group. Let $f,\ g$ be two $G$-equivariant Lipschitz maps from $X$ to $X'.$ A continuous homotopy  $F(t,x)\ (t\in [0,1])$ between $f$ and $g$ is said to be strongly Lipschitz if the following conditions are satisfied
    \begin{enumerate}
    \item $F(t,x)$ is a $G$-equivariant coarse map from $X$ to $X'$ for each $t;$
    \item $d(F(t,x),F(t,y))\leq Cd(x,y)$ for all $x,y\in X$ and $t\in [0,1],$ where $C$  is a constant (called the Lipschitz constant of $F$);
    \item $F(0,x)=f(x)$ and $ F(1,x)=g(x)$ for all $x\in X.$
    \end{enumerate}
\end{definition}
Now we recall the notion of $G$-equivariant strongly Lipschitz homotopy equivalence. 
\begin{definition}
     Let $X$ and $X'$ be two complete manifolds with cocompact, proper and isometric  $G$-actions, where $G$ is an almost connected Lie group.   We say that $X$ is $G$-equivariant strongly Lipschitz homotopy equivalent to $X'$ if there exist $G$-equivariant Lipschitz maps  $f: X\to X'$ and $g: X'\to X$ such that $g\circ f$ and $f\circ g$ are $G$-equivariant strongly Lipschitz homotopic to $id_X$ and $id_{X'}$ respectively.
\end{definition}
The proof of the following theorem is a verbatim application of the argument for \cite[Proposition 3.7]{yu1997localization} and application of the adjoint homomorphism defined before Lemma \ref{lem of adv}.
\begin{theorem}\label{thm strong lip inv}
      Let $X$ and $X'$ be two complete manifolds with cocompact, proper and isometric  $G$-actions, where $G$ is an almost connected Lie group. Suppose that $X$ is $G$-equivariant strongly Lipschitz homotopy equivalent to $X'.$ Then  $K_*(C^*_L(X)^G)$ is naturally isomorphic  to $K_*(C^*_L(X')^G).$ 

     The theorem similarly applies to the vertically representable equivariant  $K$-homology.
\end{theorem}

In \cite{nishikawa2021crossed}, Nishikawa introduced the locally compact group equivariant 
$K$-homology using localization algebras and crossed products. Nishikawa demonstrated that his  equivariant $K$-homology is isomorphic to the equivariant $KK$-homology (see \cite{kasparov1988equivariant}). We acknowledge the necessity of proving that our definition of the equivariant $K$-homology is isomorphic to the equivariant $KK$-homology. A general result for  this  lies outside the scope of our primary objective. However, for the specific case where the manifold is $G$-spin, we verify the isomorphism between our equivariant $K$-homology and equivariant $KK$-homology in Subsection \ref{subsec thom and poincare}. The general case will be addressed in future work.

\subsection{Equivariant $K$-homology class of the Dirac-type operator}\label{sub k homo class}

In this subsection, we recall the definition of the equivariant $K$-homology class of Dirac-type operators. We shall focus only on even-dimensional manifolds. 

Let  $X$ be a complete even-dimensional Riemannian manifold with a cocompact, proper and isometric $G$-action, where $G$ is an almost connected Lie group. Let $D$  be a graded Dirac-type  operator acting on the $L^2$ sections of a $G$-equivariant bundle $E\to X.$   

Let $\psi:(-\infty,\infty)\to [-1,1]$ be an odd smooth function where 
\begin{equation}\label{eq characteristic condition}
     \lim\limits_{x\to \pm \infty} \psi(x)=\pm 1,
\end{equation}
and the Fourier transform    $\hat{\psi}$ has compact support.
Then, keeping in mind that $D$ is $\mathbb Z_2$-graded, we may write 
\[
\psi\left(\frac{D}{t}\right)=\begin{pmatrix}
    0 & U_t\\
    V_t & 0 
\end{pmatrix},\quad  t\in [1,\infty).
\]
Applying inverse Fourier transform we have
\[
\psi\left(\frac{D}{t}\right)=\int_{-\infty}^{\infty} \hat{\psi}(s)e^{2\pi i s \frac{D}{t}} d s,
\]
and it follows that 
\[\lim\limits_{t\to \infty} \left[\text{propagation } \psi\left(\frac{D}{t}\right) \right]=0. \]

Now we define the following element
\[
P^{L^2(E)}_L[D](t)=W_t\begin{pmatrix}
    1 & 0 \\
    0 & 0
\end{pmatrix}W^{-1}_t, \quad t\in [1,\infty)
\]
where 
\[
W_t=\begin{pmatrix}
    1 & U_t\\
   0 & 1
\end{pmatrix}\begin{pmatrix}
    1 & 0\\
    -V_t & 1
\end{pmatrix}\begin{pmatrix}
    1 & U_t\\
    0 & 1
\end{pmatrix}\begin{pmatrix}
    0 & -1\\
    1 & 0
\end{pmatrix}, \quad t\in [1,\infty)
\]
Through direct computation, one obtains for $t\in [1,\infty)$:
\begin{equation}\label{eq pre k of d}
P^{L^2(E)}_L[D](t)=\begin{pmatrix}
    1-(1-U_{t}V_{t} )^2 & (2-U_{t}V_{t} )U_{t} (1-V_{t}U_{t} )\\
    V_{t} (1-U_{t}V_{t} ) & (1-V_{t}U_{t} )^2
\end{pmatrix}.
\end{equation}
At this stage, for every $t\in [1,\infty),$ $P^{L^2(E)}_L[D](t)$ lies in $M_2(C^*(X, L^2(E))^G)^+.$ To generate an element in $M_2(C^*(X)^G)^+,$ we embed $P^{L^2(E)}_L[D](t)$ into $M_2(B(L^2(E)\otimes L^2(G))=M_2(B(H(E)))$ by the map defined as follows. 
Let $\chi\in C^{\infty} (X) $ be a cutoff function whose support has compact intersections with all $G$-orbits and satisfies  
\[
\int_G \chi(gm)^2 d g=1, \quad  \forall m\in X.
\]
Let $j: L^2(E)\to L^2(E)\otimes L^2(G)$ be a map defined by
\[
j(s)(x,g)=\chi(g^{-1}x)s(x), s\in L^2(E), x\in X, g\in G,
\]
let $j^*$ be its adjoint.
We now define 
\begin{equation}\label{eq k homo repre of d}
   P_L[D](t)= j P_L^{L^2(E)} j^* \in M_2(B(L^2(E)\otimes L^2(G))), \quad\forall t\in [1,\infty).
\end{equation}

Thus 
\[
[t\mapsto P_L[D](t)]-\left[ \begin{pmatrix}
    1 & 0 \\
    0 & 0
\end{pmatrix}\right]
\]
defines a class in $K_0(C^*_L(X)^G).$  This class is referred to as the $G$-equivariant $K$-homology class of $D,$ denoted by $[D].$

Let $\pi :F\to  X$ be a $G$-equivariant Riemannian vector bundle. Let $\pi': E\to X$ be a $G$-equivariant Riemannian vector bundle equipped with a left $\text{Cliff}(F)$ representation. Specifically, $E=E^+\oplus  E^-,$ for any $v_x\neq 0 \in F_x$ at $ x\in X,$ the action of $v$ on $E_x$ can be expressed as 
\[
\begin{pmatrix}
   0 & v_x^+\\
   v_x^- & 0  
\end{pmatrix},
\]
according to the decomposition $E_x=E^+_x\oplus E_x^-,$ where $v_x^+$ and $v_x^-$ are both isomorphisms. Following Definition 2.8.5, Construction 2.9.15, and Theorem 2.9.16 of \cite{yu20book}, we recall a construction to define a $K$-theory class in $K_G^0(F)$ for the  Clifford  action of $F$ on $\pi^* E.$  Let $h$ be an odd non-decreasing real function on $\mathbb{R}$ such that 
\[
h(t)=1,\quad  \forall t>1.
\]
For $v\in \Gamma(F)$ with $\|v_x\|\geq 1, \forall x\in X,$ with respect to the grading, we have 
\[
h(c(v))=\begin{pmatrix}
    0 & \alpha \\
    \beta & 0
\end{pmatrix}.
\]
We define 
\begin{equation}\label{eq def of pv}
    P[v]=Z \begin{pmatrix}
    1 & 0\\
    0 & 0
\end{pmatrix}Z^{-1},
\end{equation}
where 
\[
Z=\begin{pmatrix}
    1 & \alpha \\
    0 & 1
\end{pmatrix}\begin{pmatrix}
    1 & 0 \\
    -\beta & 1
\end{pmatrix}\begin{pmatrix}
    1 & \alpha \\
    0 & 1
\end{pmatrix}\begin{pmatrix}
    0 & -1 \\
    1 & 0
\end{pmatrix}.
\]
It follows that  
\[
p[v]=\begin{pmatrix}
    1-(1-\alpha\beta )^2 & (2-\alpha\beta )\alpha (1-\beta\alpha )\\
    \beta (1-\alpha\beta ) & (1-\beta\alpha )^2
\end{pmatrix}.
\]
Then the formal difference
\[
[p[v]]-\left[ \begin{pmatrix}
    1 & 0 \\
    0 & 0
\end{pmatrix}\right]
\]
defines a class in $K^0_G(F),$ which we denote by $[v].$  In fact, it is well known that this class corresponds to the triple
\[[\pi^* E^+, \pi^*E^-, c(v) ].\]

\subsection{The pairing}\label{sub cap product}

In this subsection, we revisit the construction of the pairing between equivariant $K$-homology and $K$-theory introduced by Willett and Yu  \cite{yu20book}.  

Let  $X$ be an even-dimensional complete Riemannian manifold with a cocompact, proper and isometric $G$-action, where $G$ is an almost connected Lie group.  We define the pairing between $K_0(C^*_L(X)^G)$ and $K_G^0(X)$ as stated in the following theorem.  Note that every class of $K_G^0(X)$ can be represented by projections with $G$-compact support and, consequently, by bounded and uniformly continuous elements in $M_{\infty}(C(X)).$ 

\begin{theorem}\label{thm the cap product}
For any $[t\mapsto P(t)]\in K_0(C^*_L(X)^G)$ and $[p]\in K_G^0(X),$
where $t\to P(t)$ and $p$ are both idempotents, the path 
\[
t\mapsto P(t+T)p,\quad  t\in [1,\infty)
\]
is an almost idempotent when $T$ is a sufficiently large positive number, thus defines a class 
\[
[t\mapsto P(t+T)p]\in K_0(C^*_L(X)^G).
\]
Then one can define the pairing
\[
K_0(C^*_L(X)^G)\otimes K_G^0(X)\to K_0(C^*_L(X)^G)
\]
by 
\[
[t\mapsto P(t)]\otimes [p]=[t\mapsto P(t+T)p].
\]
It is straightforward to see that the definition is well defined. 
\end{theorem}

\begin{proof}
  We need to show that there exist  $\epsilon<\frac{1}{10}$ and $T>0$ such that for 
  \[\|(P(t+T)p)^2-P(t+T)p\|<\epsilon,\quad  \forall t>0.\]
Since   
\[
(P(t+T)p)^2-P(t+T)p=(P(t+T)p)^2-P^2(t+T)p^2=-P(t+T) [P(t+T),p]p,
\]
it suffices to show that 
\[
\lim\limits_{t\to +\infty} \|[P(t+T),p]\|=0,
\]
which is established in \cite[Lemma 6.1.1]{yu20book}. 

Similarly one can define the pairing between $G$-equivariant  vertically representable $K$-homology and $G$-equivariant $K$-theory.
\end{proof}

In the following, we will also consider the pairing:
\[
K_0(C^*_L(X)^G)\otimes K^0(X) \to K_0(C^*_L(X))
\]
whose definition is entirely analogous to the one given in Theorem \ref{thm the cap product}.

\subsection{Poincar\'e duality}\label{subsec thom and poincare}

In this subsection we establish the Poincar\'e duality between equivariant $K$-homology and equivariant $K$-theory for  $G$-$\text{spin}$  manifolds, where an almost connected Lie group $G$ acts on the manifold cocompactly, properly and isometrically.

Let $G$ be an almost connected Lie group acting properly, cocompactly and isometrically on $X,$ which is a $G$-$\text{spin}$ manifold. Denote by $S(T^*X)$  the spin bundle associated to $T^*X,$ and let $\D: L^2(S(T^*X))\to L^2(S(T^*X))$ be the densely defined Dirac operator. 

\begin{definition}\label{def times D}
Let $ P[{\D}]: [1,+\infty)\mapsto (C^*_L(X)^G)^+$ be a representative of 
\[[\D]\in K_0(C_L^*(X)^G).\]
Define the homomorphism
\[
\otimes [\D] : K^0_G(X)\to K_0(C_L^*(X)^G),
\]
by 
\[
[p]\otimes [\D] = [t\mapsto P_L[\D](t+T) p]=[t\mapsto pP_L[\D](t+T)], \quad  \forall [p]\in K^0_G(X)
\]
where $T>0$ is sufficiently large, and $P_L[\D]$ is as defined in Subsection \ref{sub k homo class}. 
\end{definition}
It is straightforward to verify that this definition is independent of the choice of the representative $P_L[\D]$ and $T.$ 

As was shown in \cite{kasparov1988equivariant}, there is a constant  $r>0,$ such that for any  $x,y\in X$ with the distance $\rho(x,y)<r,$  there is a unique geodesic segment joining $x$ and $y.$ For any $x\in X,$ let $U_x$ denote the ball of radius $r$ centered at $x.$ Define a differential 1-form on  $U_x$ by 
\[
\theta_x(y):= \frac{\rho(x,y)}{r}(d_y\rho)(x,y), \quad \forall y \in U_x
\]
where  $\rho(x,y)$ stands for the distance function. Due to the Clifford representation,  $\theta_x$ acts as a bounded operator on $L^2(S(T^*X)).$ Moreover,  $1-\theta_x^2$ is a continuous function in $C(X).$ Consider $\theta_X$ as the family of $\{\theta_x\}_{x\in X}.$ 

\begin{definition}
Define the homomorphism
\[
\otimes [\theta_X] : K_0(C_L^*(X)^G)\to K_G^0(X)
\]
by 
\[
[P]\otimes [\theta_X]\mapsto [ P(T)p [\theta_X]  ], \forall [P]\in  K_0(C_L^*(X)^G),
\]
where $T$ is a sufficiently large positive number, and $p [\theta_X] $ is a representative of the class $[\theta_X]\in K_G^0(X)$ as defined in Subsection \ref{sub k homo class}.  
Note that for any $x\in X,$ $P(T)p [\theta_X]$ is an almost projection in $B(L^2(S(T^*X))\otimes L^2(G)),$ thus $P(T)p[\theta_X]$ defines a family of almost projections with coefficients depending on $x\in X,$ which represents a class in $K_G^0(X).$ It is straightforward to verify that this definition is independent of the choice of $p[\theta_X]$ and $T.$ 
\end{definition}

The following lemma is inspired by \cite{kasparov1988equivariant,yu20book}.

\begin{proposition}\label{prop dtheta for mfld}
The homomorphism $\otimes [\theta_X]$ is the right inverse of the homomorphism $\otimes [\D],$ i.e., for all $[p]\in K_G^0(X),$ 
\[
[p]  \otimes [\D] \otimes [\theta_X]= [p]\otimes ([\D] \otimes \theta_X)=[p]\otimes [1_X],
\]
where $[1_X]$ stands for the class represented by the rank $1$ trivial projection in $M_{\infty}(C(X)).$
\end{proposition}

    \begin{proof}
Let 
\[ P_L[\D]: [1,\infty)\mapsto (C^*(X)^G)^+\] be a representative of 
\[ [\D] \in K_0(C_L^*(X)^G).\]
It suffices to prove that 
\[[P_L[\D](T)p[\theta_X]]=[1_X]\in K_G^0(X).\]
However, as shown in  \cite[Lemma 9.33]{yu20book},  
\[[P_L[\D](T)p[\theta_X]]=
[P_L[\D\hat{\otimes} 1+1\hat{\otimes}\theta_X](T)], \]
where 
\[\D\hat{\otimes} 1+1\hat{\otimes}\theta_X\]
is a densely defined operator on 
\[L^2(S(T^*X)\hat\otimes S(T^*X))=L^2(\text{Cliff}(T^*X)).\]
Note that $\D\hat{\otimes} 1$ acts on $L^2(\text{Cliff}(T^*X))$ as the de Rham operator, and the action of $1\hat{\otimes}\theta_X$  on $L^2(\text{Cliff}(T^*X))$ is induced by the right Clifford multiplication twisted by the grading
\[w\to (-1)^{\partial w}w \theta_X, \quad \forall \text{ homogeneous } w\in \text{Cliff}(T^*(X)).\]
However, \cite[Theorem 9.3.5]{yu20book} shows that 
\[[P_L[\D\hat{\otimes} 1+1\hat{\otimes}\theta_X](T)]=[1_X]\in K_G^0(X). \]
\end{proof}

The following proposition can be viewed as the dual of Proposition \ref{prop dtheta for mfld}.

\begin{proposition}\label{prop thetad for mfld}
The homomorphism $\otimes[\D]$ is the right inverse of the homomorphism $\otimes [\theta_X],$ i.e.
\[
[P]  \otimes [\theta_X]  \otimes [\D]= [P], \quad \forall [P]\in K_0(C^*_L(X)^G).
\]
\end{proposition}
\begin{proof}
In fact, we have 
\[
[P]  \otimes [\theta_X] \otimes [\D] 
= [P(T) p [\theta_X] ]\otimes [\D] 
=[P(T) p [\theta_X]P_L[\D](t+T)],
\]
where $P_L[\D](t+T)$ (resp. $ p [\theta_X] $) is a representative of $[\D]\in K_0(C^*_L(X^G)$ (resp. $[\theta_X]\in K_G^0(X)$) as defined in Subsection \ref{sub k homo class}.

We also have the following equality
\[ [P(T)  p [\theta_X]  P_L[\D](t+T)]=[P(t+T) p[\theta_X]P_L[\D](t+T)].\]

Recall that $p [\theta_X]$ is a family of almost projections corresponding to $x\in X.$ For each $x\in X,$ we denote the almost projection in  $p [\theta_X]$ corresponding to $x$  by $h(x,y).$
Let $K_t(x,y)$ be the kernel of the integral operator representing $P(t+T),$ and $K_{\D,t}(x,y)$ be the kernel of the integral operator representing $P_L[\D](t+T).$ Then 
\[
P(t+T) p[\theta_X]  P_L[\D](t+T)
\]
is a path of locally compact almost projections with finite propagation on $L^2(X)\otimes L^2(X)$ defined by kernel 
\[
K_t(z,y) h(x,y) K_{\D,t} (x,w).
\]
Thus, the swap of variables $x$ and $y$ induces a unitary equivalence between 
\[P(t+T) p [\theta_X] P_L[\D](t+T)\]
and 
\[P(t+T)(  p [\theta_X]  P_L[\D](t+T) ).\]
This implies that 
\begin{eqnarray*}
&&[P(t+T) p [\theta_X]   P_L[\D](t+T)]\\
&=&[P(t+T) (p [\theta_X]   P_L[\D](t+T))]\\
&=&[P(t+T)]\otimes [ p [\theta_X]  P_L[\D](t+T)]\\
&=&[P(t+T)]\otimes [1_X]\\
&=& [P].
\end{eqnarray*}
\end{proof}

Combining Propositions \ref{prop dtheta for mfld} and \ref{prop thetad  for mfld}, one can see that 
\begin{theorem}\label{thm pd for mfld}
   In the case where $G$ is an almost connected Lie group and $X$ is a $G$-spin manifold with an isometric, proper and cocompact $G$-action, the homomorphism
    \[\otimes [\D] : K^0_G(X)\to K_0(C_L^*(X)^G)\]
    is an isomorphism with inverse 
    \[
    \otimes [\theta_X] : K_0(C_L^*(X)^G)\to K_G^0(X).
    \]
\end{theorem}

As a corollary, we have 
\begin{theorem}\label{thm qualify our k homo}
     In the case where $G$ is an almost connected Lie group and $X$ is a $G$-spin manifold with an isometric, proper and cocompact $G$-action, we have 
    \[
    K_0(C^*_L(X)^G)\cong KK_0^G(\mathbb{C}, C_{0}(X)).
    \]
\end{theorem}

Let $G$ be an almost connected Lie group and $X$ be a $G$-spin manifold with an isometric, proper and cocompact $G$-action. Let $E\to X$ be a $G$-equivariant vector bundle over $X,$ such that $E$ is $G$-spin as a manifold.   Let ${\D}^E$ be the Dirac operator on $E.$ Although $E$ is not a cocompact $G$-manifold, $\theta_E$ is as well defined  as  $\theta_X.$ We emphasize that for the definition of $\theta_E,$ $E$ is considered as a manifold.

Similar to above, one can define the following two homomorphisms
\[
\otimes [\D^E] : K_G^0(E)\to K_0(RVC^*_L(E)^G),\ \  \otimes [\theta_E] :   K_0(RVC^*_L(E)^G)\to K_G^0(E).
\]
Here $\D^E$ is the Dirac operator on $E,$ not the Dirac operator on $X$ twisted with $E.$

By exactly the same argument, we have 
\begin{theorem}\label{thm pd for bundle}
    The homomorphism
\[
\otimes [\D^E] : K_G^0(E)\to K_0(RVC^*_L(E)^G)\]
is an isomorphism with inverse 
\[\otimes [\theta_E] :   K_0(RVC^*_L(E)^G)\to K_G^0(E).
\]

\end{theorem}

\begin{proof}
For any $[p]-[q]\in K_G^0(E)$,  $p-q$ has vertically compact support. Hence, by definition of the pairing, $([p]-[q])\otimes [\D^E]$ is also represented by formal difference with vertically compact support. Thus, $\otimes[\D^E]$ is a well defined homomorphism from $K_G^0(E)$ to $K_0(RVC^*_L(E)^G).$ By the same reason, one can see that $\otimes [\theta_E]$ is also well defined. 

Now for $[p]-[q]\in  K_G^0(E)$, such that $p-q$ has vertically compact support, the proof for Proposition \ref{prop dtheta for mfld} applies verbatim to give  
    \[
    ([p]-[q])\otimes [\D^E]\otimes [\theta_E]=([p]-[q])\otimes [1_E].
    \]
    Note that $[1_E]\in (K_G^0(E))^+$ is the unit. On the other hand, for $[P]\in K_0(RVC^*_L(E)^G)$, by the proof for Proposition \ref{prop thetad for mfld}, we have 
    \[
    [P]\otimes [\theta_E]\otimes [\D^E]=[P].
    \]
    These imply that $\otimes [\D^E]$ and $\otimes [\theta_E]$ are isomorphisms that are inverses of each other. The proof is then completed.
\end{proof}

We will later apply the following standard induction isomorphism to identify
the equivariant \(K\)-theory of the vector bundle over an orbit with that of its
 fiber.

Let \(H\leq G\) be a compact subgroup, and let \(V\) be a finite-dimensional
orthogonal \(H\)-representation. Set
\[
E:=G\times_H V\longrightarrow G/H,
\qquad [g,v]\longmapsto gH.
\]
Thus \(E\) is the \(G\)-equivariant vector bundle induced from the
\(H\)-representation \(V\).

\begin{theorem}[Induction isomorphism]
\label{thm:induction-k-theory}
The induction map gives a natural isomorphism
\[
\operatorname{Ind}_H^G\colon
K_H^0(V)\xrightarrow{\cong}K_G^0(E).
\]
\end{theorem}


\section{Localization at $K$-homology  level }\label{sec main proof}

In this section, we present the proof of the main result, namely Theorem \ref{thm main}. The proof is carried out in two steps. First, we localize the equivariant $K$-homology class of the de Rham operator to the tubular neighborhoods of zero submanifolds of a Morse-Bott differential 1-form by perturbing the operator with this field. Next, we compute the equivariant $K$-homology class of the de Rham operator using the Poincar\'e duality established in Subsection \ref{subsec thom and poincare} and equivariant $K$-theory.

\subsection{Localization of the class}

In this subsection, we use the Witten perturbation to localize the equivariant $K$-homology class of the de Rham operator to the tubular neighborhoods of zero submanifolds of a Morse-Bott differential 1-form.

Before proceeding, let us recall the definition of a Morse–Bott differential 1-form (\cite{equivariantdifferentialtopology}).

\begin{definition}\label{def morsebott covector filed}
     A smooth function $f:X\rightarrow \mathbb{R}$ is
  Morse–Bott if its critical points form a disjoint union of connected submanifolds $Y_{1},\dots, Y_{r},$ and for every  $x\in Y_{i},$ the Hessian of $f$ is nondegenerate on all subspaces of $T_{x}X$ intersecting $T_{x}Y_{i}$ transversally. The differential 1-form $df$ associated with any Morse–Bott function $f$ is referred to as a Morse–Bott differential 1-form.
\end{definition}
 \begin{remark}
    In \cite{equivariantdifferentialtopology}, Wasserman demonstrated that every finite-dimensional $G$-manifold $X,$ where both $G$ and $X$ are compact, admits equivariant Morse-Bott functions. Consequently, we assert that such manifolds also admit an equivariant Morse-Bott differential 1-form. 
    \end{remark}

Then, we define an equivariant Morse-Bott differential 1-form for a noncompact Lie group:
\begin{definition}
A differential 1-form \( A \) is called an \emph{equivariant Morse--Bott differential 1-form} for a noncompact Lie group \( G \) if there exists a \( G \)-invariant Morse--Bott function \( f \) such that \( A = df \), and the following conditions hold:
\begin{itemize}
    \item The critical set of \( f \) consists of submanifolds \( \{Y_i\} \), each of which is \( G \)-invariant.
    \item For each \( Y_i \), there exists a \( G \)-invariant tubular neighborhood modeled by a \( G \)-equivariant vector bundle \( E_i \to Y_i \), together with a \( G \)-equivariant embedding \( h: E_i \to X \), such that:
    \begin{itemize}
        \item The bundle \( E_i \) admits an orthogonal \( G \)-invariant splitting \( E_i = E_i^+ \oplus E_i^- \),
        \item There exists a \( G \)-invariant open neighborhood \( U_i \subset E_i \) of the zero section such that for all \( e = (e_+, e_-) \in U_i \),
        \[
        f(h(e)) = c - \|e_-\|^2 + \|e_+\|^2,
        \]
        where \( c = f|_{Y_i} \) is constant, and the norms are taken with respect to a \( G \)-invariant Riemannian metric on \( E_i \).
    \end{itemize}
\end{itemize}

\end{definition}

Now we consider the existence of the equivariant Morse-Bott differential 1-form we defined. More generally, in the examples we will present below, the critical submanifolds of the equivariant covector fields are some orbits generated by single points.
\begin{proposition}
\label{exist}
    Let \(G\) be an almost connected Lie group, and assume that \(G\) contains a compact subgroup isomorphic to \(S^1\). We identify this subgroup with \(S^1\) and let it act on \(S^2\) by rotations about the vertical axis. Consider the \(2\)-sphere \(S^2 \subset \mathbb{R}^3\), positioned so that its south pole is at the origin. The height function
\[
f:S^2 \longrightarrow \mathbb{R}, \qquad f(x,y,z)=z,
\]
is invariant under the \(S^1\)-action and therefore induces a well-defined smooth function
\[
\begin{aligned}
F:G\times_{S^1}S^2 &\longrightarrow \mathbb{R},\\
[g,s] &\longmapsto f(s).
\end{aligned}
\]
Consequently, \(dF\) is a \(G\)-equivariant Morse--Bott differential \(1\)-form on \(G\times_{S^1}S^2\). Its critical submanifolds are precisely the \(G\)-orbits through the points represented by the north and south poles of \(S^2\).
\end{proposition}
\begin{proof}
     We investigate the action of \(S^1\) on \(S^2\), which is the rotation about the z-axis. It can be seen that \(f\) is an equivariant Morse-Bott function under this action, and its critical points are the south pole and the north pole. It implies that  \(F\) is well defined. 
     
     Let \(x_0\) and \(x_1\) denote the south pole and the north pole, respectively. We can conclude that the critical submanifolds of \(F\) are  \(G\times_{S^1}\{x_{0}\}\) and  \(G\times_{S^1}\{x_{1}\}\). Since \(G\times_{S^1} S^2\) has the structure of a fiber bundle with base space \(G/S^1\) and fiber \(S^2\), the function \(F\) essentially depends only on the points in the fiber \(S^2\). This implies that the non-degeneracy of the Hessian  of \(f\) at the south pole and the north pole can lead to the non-degeneracy of the Hessian of \(F\) when restricted to \(T_{x}(G\times_{S^1}S^2)/T_{x}(G\times_{S^1}\{x_{i}\})\) for any \(x\in G\times _{S^1}\{x_{i}\}\) (\(i=0,1\)).

      Finally, for each \( G\times_{S^1}\{x_{i}\} \), there exists a \( G \)-invariant tubular neighborhood modeled by a \( G \)-equivariant nomoarl bundle \(N_{i}\) of \( G\times_{S^1}\{x_{i}\} \) , together with a \( G \)-equivariant embedding \( h: N_{i} \to G\times_{S^1}S^2\) induce by the exponential map such that: The bundle \( N_i \) admits an orthogonal \( G \)-invariant splitting \( N_i = N_i^+ \oplus N_i^- \). In fact, we have \(N_0 = N_0^+\) and \(N_1 = N_1^-\).
         There exists a \( G \)-invariant open neighborhood \( U_i \subset N_i \) of the zero section such that for all \( e = (e_+, e_-) \in U_i \),
        \[
        F(h(e)) = c - \|e_-\|^2 + \|e_+\|^2,
        \]
        where \( c = f|_{G\times_{S^1}\{x_{i}\}} \) is constant, and the norms are taken with respect to a \( G \)-invariant Riemannian metric on \( N_i \).

  Thus, we have verified that \(dF\) is an equivariant Morse-Bott differential 1-form for the almost connected Lie group \(G\) and the critical submanifolds of \(dF\) are  some orbits generated by single points.
\end{proof}

\begin{remark}
    In fact, for any almost connectd Lie group \(G\) and one of its compact subgroups \(H\), if there exists a smooth compact manifold \(S\) without boundary admitting a smooth \(H\)-action, then there exists an equivariant Morse-Bott differential 1-form on \(G \times_H S\), and the critical submanifolds of the equivariant covector fields are still some orbits generated by single points. More generally, this construction applies to every smooth manifold without boundary equipped with a proper and cocompact action of an almost connected Lie group \(G\). Indeed, by the global slice theorem, there exist a maximal compact subgroup \(K\subset G\) and a compact \(K\)-manifold \(S\) such that
    \[
       M\cong G\times_K S
     \]
    as \(G\)-manifolds. The construction is similar to the case of \(G \times_{S^1} S^2\).  First, construct an \(H\)-equivariant Morse-Bott function \(\psi\) on \(S\), whose critical submanifolds are orbits generated by single points. The existence of such a function \(\psi\) is guaranteed by the discussion in  \cite{equivariantdifferentialtopology}. Then, \(\psi\) induce a \(G\)-equivariant smooth function:
    \begin{align*}
        \Psi: &G\times_{H} S\rightarrow \mathbb{R}\\
        &[g,s]\mapsto \psi(s)
        \end{align*}
        such that \(d\Psi\) is an equivariant Morse-Bott differential 1-form for almost connected Lie group \(G\) and the critical submanifolds of \(d\Psi\) are  some orbits generated by single points. The proof of this conclusion is similar to that of Proposition \ref{exist}.
\end{remark}

Let $G$ be an almost connected Lie group acting properly, cocompactly and isometrically on a complete Riemannian manifold $X.$ Assume that $A=d f$ is a $G$-invariant Morse–Bott differential 1-form over $X,$ whose zero subset $Z$ consists of $G$-orbits 
\[
Z=Z_1\cup Z_2 \cup \cdots \cup Z_k.
\]

Choose $\delta>0$ satisfying the following conditions 
\begin{enumerate}
\item $\|A(x)\|\geq 1, \forall x\in X\backslash N_{\delta}(Z),$ where  $N_{\delta}(Z)$ is the $\delta$-tubular neighborhood of $Z.$
\item $N_{1000\delta}(Z_i)\cap N_{1000\delta}(Z_j)=\emptyset, \forall 1
    \leq i\neq j\leq k, $ where $N_{1000\delta}(Z_i)$ is the $1000\delta$-tubular neighborhood of $Z_i.$

\end{enumerate}

Note that the zero subset $Z$ consists of finitely many $G$-orbits that can be separated  by some $\delta>0$ is a consequence of the action being proper, cocompact and isometric.

Consider the right Clifford multiplication by $A$ twisted by grading
\[
w\to (-1)^{\partial w}wA, \quad \forall \text{ homogeneous } w\in \text{Cliff}(T^*X).
\]
This defines a bounded operator on $L^2(\text{Cliff}(T^*X)),$ which we continue to denote by $A.$

Let $D$ be the de Rham operator on $X$. Then it is straightforward to verify that 
\[
[D]=[D+\lambda A]\in K_0(C^*_L(X)^G), \quad \lambda >0.
\]
Thus, to compute $[D],$ it suffices to compute $[D+\lambda A].$

The following lemma is from \cite{liu2024equivariant}.

\begin{lemma}
    Let $\lambda$ be a positive number. Then for all $t>0,$ as essentially selfadjoint operators, 
    \[(t^{-1}D+\lambda A)^2 \geq t^{-2}D^2+\lambda^2 \inf\limits_{x\in X} \|A(x)\|^2 -Ct^{-1}\lambda, \]
    where $C$ is a positive number depending only on the manifold $X$ and the differential 1-form $A.$ In particular, restricting to $X\backslash N_{\delta}(Z),$ 
       \[(t^{-1}D+\lambda A)^2 \geq \lambda^2 -Ct^{-1}\lambda.\]
      
\end{lemma}
Thus we can choose $\lambda>0$ properly such that $(t^{-1}D+\lambda A)^2$ is bounded below by $1$ uniformly for $t\in [1,\infty).$ Now let $\psi:(-\infty, \infty)\to [-1,1]$  be a smooth, odd function satisfying  conditions in line \eqref{eq characteristic condition}. We further choose $\psi$ properly such that 
\[
1-\psi^2(s)<\frac{1}{10000}, \quad \forall \|s\|\geq 1,
\]
and $T_0\in \mathbb{R}_+$ sufficiently large such that
\[
\text{propagation of } \psi(t^{-1}D+\lambda A)<\delta, \quad \forall t\in [T_0,+\infty).
\]
 The latter can be obtained by demanding that the support of $\hat{\psi}$ lies in $[-T_0\delta, T_0 \delta).$

Assume that under the $\mathbb{Z}_2$ grading by the parity of forms, 
\[\psi(t^{-1}D+\lambda A)=\begin{pmatrix}
    0 & U_{t,\lambda}(t)\\
    V_{t,\lambda}(t) & 0 
\end{pmatrix}. \]
As constructed in Subsection \ref{sub k homo class}, the $G$-equivariant $K$-homology class of $D+\lambda A$ is represented by 
\[
[t\mapsto P_{\lambda}(t)]-\left[\begin{pmatrix}
    1 & 0 \\
    0 & 0
\end{pmatrix}\right]
\]
where 
\begin{equation}\label{eq Pt of D plus lambda A}
P_{\lambda}(t)=\begin{pmatrix}
    1-(1-U_{t,\lambda}V_{t,\lambda} )^2 & (2-U_{t,\lambda}V_{t,\lambda} )U_{t,\lambda} (1-V_{t,\lambda}U_{t,\lambda} )\\
    V_{t,\lambda} (1-U_{t,\lambda}V_{t,\lambda} ) & (1-V_{t,\lambda}U_{t,\lambda} )^2
\end{pmatrix}.
\end{equation}

In the following, we decompose 
\[L^2(\text{Cliff}(T^*X)) = L^2(N_{50\delta }(Z)) \oplus L^2(N_{100\delta }(Z)\backslash N_{50\delta }(Z) ) \oplus L^2(N^c_{100\delta }(Z) ) ,
\]
where $L^2(N_{50\delta }(Z)),$  $L^2(N_{100\delta }(Z)\backslash N_{50\delta }(Z) ), $ and $L^2(N^c_{100\delta }(Z) ) $ represent the restriction of $L^2(\text{Cliff}(T^*X)) $ to $N_{50\delta }(Z),$  $N_{100\delta }(Z)\backslash N_{50\delta }(Z),$ and the complement $ N^c_{100\delta }(Z),$ respectively.

\begin{lemma}\label{lem of roe}
    Restricting to 
    \[
    L^2(N^c_{50\delta }(Z) ) =  L^2(N_{100\delta }(Z)\backslash N_{50\delta }(Z) ) \oplus L^2(N^c_{100\delta }(Z) ),
    \]
    we have 
    \[
    \|1-\psi^2(t^{-1}D+\lambda A)\|\leq \frac{1}{10000}
    \]
    uniformly for $t\in [T_0,\infty)$
\end{lemma}

\begin{proof} The proof follows essentially from \cite[Lemma 2.5]{roe2016positive}.
    Since  $(t^{-1}D+\lambda A)^2$  is uniformly bounded below by 
$1$ on  $N^c_{40\delta }(Z)$ for all $t\in [T_0,\infty),$ it admits a Friedrichs extension on the Hilbert space  $L^2(N^c_{40\delta }(Z) ).$ This extension, which we denote by $E_t,$ is also bounded below by $1.$  

Observe that $g=1-\psi^2$ is an even function, thus 
\[g(x)=2\int_0^\infty \hat{g}(s)\cos (sx)ds. \]
Note that $\hat{g}$ has compact support $[-2T_0\delta,2T_0\delta]$.  For all $0\leq s\leq 2T_0\delta,$ and $t\geq T_0,$ we have 
\[
\cos(s(t^{-1}D+\lambda A))=\cos (s\sqrt{E_t}) 
\]
restricting on $L^2(N^c_{50\delta }(Z) ).$ 
This implies that for $t\geq T_0,$
    \begin{eqnarray*}
        g(t^{-1}D+\lambda A)&=&\int_0^{T_0\delta} \hat{g}(s)\cos (s(t^{-1}D+\lambda A))ds\\
        &=&\int_0^{T_0\delta} \hat{g}(s)\cos (s\sqrt{E_t})ds\\
        &=&  g(\sqrt{E_t})
    \end{eqnarray*}
    when restricted on $L^2(N^c_{50\delta }(Z) ).$ The lemma then follows from the facts that $E_t$ is uniformly bounded below by $1$ and that 
    \[
1-\psi^2(s)<\frac{1}{10000}, \quad \forall \|s\|\geq 1.
\]
\end{proof}
As a corollary, we have 
\begin{corollary}\label{cor the cut off}
    Consider $P_{\lambda}(t)$ defined in \eqref{eq Pt of D plus lambda A}. Then, with respect to the decomposition 
    \[
    L^2(\text{Cliff}(T^*X)) = L^2(N_{100\delta }(Z)) \oplus L^2(N^c_{100\delta }(Z) ),
    \]
    we have 
    \[
    \left\| P_{\lambda}(t)- \begin{pmatrix}
        \overline{P}_{\lambda}(t) & 0 \\
        0 & \begin{pmatrix}
            1 & 0 \\
            0 & 0
        \end{pmatrix}
    \end{pmatrix} \right\|\leq \frac{1}{1000},\quad  \forall t\in [T_0, \infty)
    \]
    where $\overline{P}_{\lambda}(t)$ is an almost idempotent in  $(RVC^*_L(N_{50\delta}(Z)))^+$.
\end{corollary}
\begin{proof}
    By definition, the propagation of $P(t)$ is less than $6\delta$ for every $t\in [1,\infty),$ thus with respect to the decomposition
    \[
        L^2(\text{Cliff}(T^*X)) 
        =L^2(N_{50\delta }(Z)) \oplus L^2(N_{100\delta }(Z)\backslash N_{50\delta }(Z) ) \oplus L^2(N^c_{100\delta }(Z) ),
    \]
    $P_{\lambda}(t)$ equals 
    \[
    \begin{pmatrix}
        P_{11}(t) & P_{12}(t) & 0\\
        P_{21}(t) & P_{22}(t) & P_{23}(t)\\
        0 & P_{32}(t) & P_{33}(t)
    \end{pmatrix}.
    \]
    However, Lemma \ref{lem of roe} shows that
    \[
\left\|
\begin{pmatrix}
    P_{22}(t) & P_{23}(t)\\
    P_{32}(t) & P_{33}(t)
\end{pmatrix}-\begin{pmatrix}
    1 & 0\\
    0 & 0
\end{pmatrix}
\right\|\leq 10 \frac{1}{10000}=\frac{1}{1000},
    \]
    then the corollary follows immediately.
\end{proof}
Recall that $t\mapsto \overline{P}_{\lambda}(t)$ is a path of almost idempotents,
\[
[t\mapsto \overline{P}_{\lambda}(t)]-\left[\begin{pmatrix}
    1 & 0 \\
    0 & 0
\end{pmatrix}\right]
\]
representing a class in $K_0(RVC^*_L(N_{100\delta}(Z) )^G).$ Then we have 
\begin{theorem}\label{thm the cut off}
    The $G$-equivariant $K$-homology class 
    \[
    [D]=[D+\lambda A]\in K_0(C^*_L(X)^G)
    \]
    lies in the image of 
    \[
    \iota_* :  K_0(RVC^*_L(N_{100\delta}(Z))^G)\to K_0(C^*_L(X)^G),
    \]
    where $\iota_*$ is the 
homomorphism induced by the embedding
    \[
    \iota : N_{100\delta}(Z)\to X.
    \]

\end{theorem}

In the next section, we will study the representation information it carries by localizing the K-homology class of the de Rham operator.

\subsection{The computation via Poincar\'e duality}

Theorem~\ref{thm the cut off} establishes part~\textup{(1)} of
Theorem~\ref{thm main}. We now turn to the proof of
part~\textup{(2)}. The main ingredient is the Poincar\'e duality
isomorphism established in Theorem~\ref{thm pd for bundle}, applied
to the normal bundles of the components of the zero set.

We first establish the following consequence of Poincar\'e duality.

\begin{theorem}\label{thm main pd}
Let $G$ be an almost connected Lie group acting properly, cocompactly,
and isometrically on a $G$-spin manifold $X$. Suppose that there exists
a $G$-invariant Morse--Bott $1$-form $A$ on $X$ whose zero set
decomposes as
\[
Z=\bigsqcup_{i=1}^k Z_i=\bigsqcup_{i=1}^k G/H_i,
\]
where each $H_i\leq G$ is a compact subgroup. For each \(i\), let
\(E_i\to Z_i\) be the normal bundle of \(Z_i\) in \(X\), and let
\(V_i=(E_i)_{eH_i}\) be the normal fiber over \(eH_i\in G/H_i\).
Then, for each \(i\), Poincar\'e duality induces a canonical
isomorphism
\[
-\otimes [\D^{E_i}] \colon
K^0_G(E_i)
\longrightarrow
K_0(RVC^*_L(E_i)^G).
\]
Let \([A_{E_i}]\) denote the \(K\)-theory class defined by the
restriction of \(A\) to \(E_i\). Then
\[
\iota_{*}\left(
\sum_{i=1}^k [A_{E_i}]\otimes[\D^{E_i}]
\right)
=
[D]\in K_0(C^*_L(X)^G).
\]
\end{theorem}

\begin{proof}[Proof of Theorem~\ref{thm main pd}]
Recall that
\begin{eqnarray*}
K_0(RVC^*_L(N_{100\delta}(Z))^G)
&=&
\bigoplus_{i=1}^k
K_0(RVC^*_L(N_{100\delta}(Z_i))^G)
\\
&=&
\bigoplus_{i=1}^k
K_0(RVC^*_L(E_i)^G),
\end{eqnarray*}
where $E_i$ is the normal bundle of $Z_i$ in $X$. By the localization
and decomposition argument given above, the class $[D]$ therefore
decomposes as
\[
[D]
=
\bigoplus_{i=1}^k
\iota_*[D^{E_i}+\lambda A_{E_i}],
\]
where $D^{E_i}$ is the de Rham operator on $E_i$,
$A_{E_i}$ is the natural extension to $E_i$ of the differential
$1$-form $A$, and
\[
[D^{E_i}+\lambda A_{E_i}]
\in K_0(RVC^*_L(E_i)^G)
\]
is the vertically representable \(K\)-homology class of
\(D^{E_i}+\lambda A_{E_i}\).

It remains to identify each local class
\[
[D^{E_i}+\lambda A_{E_i}]
\]
under the Poincar\'e duality isomorphism. To this end, we use the
$G$-spin structure on $E_i$.

Recall that each $Z_i$ is diffeomorphic to $G/H_i,$ and that $N_{100\delta}(Z_i)$ is equivariant diffeomorphic to $E_i.$
It follows from the preceding discussion that both $Z_i$ and $E_i$
are $G$-spin manifolds. Let $S(T^*(E_i))$ denote the spinor bundle on
$E_i$, and let $\D^{E_i}$ be the Dirac operator on $E_i$. Consider the
Clifford action of $A_{E_i}$ on $S(T^*(E_i))$, together with its
extension ${A}_{E_i}$ to $L^2(S(T^*(E_i)))$. By
\cite[Lemma~9.33]{yu20book}, we have
\begin{equation}\label{eq final product}
[D^{E_i}+\lambda A_{E_i}]
=
[\D^{E_i}]\otimes[{A}_{E_i}],
\end{equation}
where $[{A}_{E_i}]$ is the class in $K_G^0(E_i)$ defined by
${A}_{E_i}$; see Subsection~\ref{sub k homo class}. Combining
\eqref{eq final product} with Theorem~\ref{thm pd for bundle} (we emphasize that we are carrying out the Poincar\'e pairing on the total space regarded as a non-cocopmact manifold, other than fiberwise pairing on a fiber manifold), we
conclude that, under the Poincar\'e duality isomorphism, the class
\([A_{E_i}]\) corresponds to
\([D^{E_i}+\lambda A_{E_i}]\). Substituting these identifications into
the above decomposition of \([D]\) proves the desired formula.
\end{proof}

To obtain the representation-theoretic formula required in
Theorem~\ref{thm main}, it remains to determine the preimage of
\([A_{E_i}]\) under the induction isomorphism
\[
\operatorname{Ind}_{H_i}^{G}\colon
K^0_{H_i}(V_i)
\longrightarrow
K^0_G(E_i).
\]
We therefore begin by giving a more explicit description of the group
\(K^0_{H_i}(V_i)\).

We recall some standard ideals and results in equivariant twisted \(K\)-theory. 

Let \(M\) be a smooth manifold equipped with a proper cocompact action of a compact Lie group \(H\), and let \(E \to M\) be a finite-rank \(H\)-equivariant real vector bundle endowed with an \(H\)-invariant Euclidean metric. The equivariant compactly supported \(K\)-theory group \(K_H^0(E)\) admits an Atiyah--Bott--Shapiro description in terms of \(H\)-equivariant \(\mathbb{Z}/2\)-graded module bundles over the complex Clifford algebra bundle \( \operatorname{Cl}(E) \longrightarrow M \),
subject to the usual stabilization and Clifford-extension relations.

Suppose, in addition, that \(E\) has even rank and is equipped with an \(H\)-equivariant \(\operatorname{Spin}^c\)-structure. In this case, \(\operatorname{Cl}(E)\) is \(H\)-equivariantly Morita equivalent to the trivial algebra bundle, and the preceding description gives rise to the usual equivariant Thom isomorphism. In general, however, the Clifford algebra bundle need not be Morita trivial. To accommodate this more general situation, one may abstract the above construction and, starting from any \(H\)-equivariant bundle \(\mathcal{A}\) of \(\mathbb{Z}/2\)-graded finite-dimensional complex algebras over \(M\), define a generalized \(K\)-theory in terms of \(H\)-equivariant \(\mathcal{A}\)-module bundles. This is the basic idea underlying twisted equivariant \(K\)-theory.

The twisted \(K\)-theory associated with \(\mathcal{A}\) depends only on the \(H\)-equivariant Morita equivalence class of \(\mathcal{A}\). Such a Morita equivalence class is called a \emph{twisting}. Under suitable assumptions on the algebra bundles under consideration, these twistings form the equivariant graded Brauer group
\[
\operatorname{GBr}_H(M).
\]

We now specialize the preceding discussion to the situation relevant to us. 
Let \(M=\mathrm{pt}\), and let
\[
\rho\colon H \longrightarrow O(V)
\]
be a finite-dimensional orthogonal real representation of the compact Lie group
\(H\). Regarding \(V\) as an \(H\)-equivariant real vector bundle over a point,
its complex Clifford algebra determines a twisting
\[
\tau(V)\in \operatorname{GBr}_H(\mathrm{pt}).
\]
Under the standard characteristic-class description of equivariant twistings,
this twisting is represented by
\[
\tau(V)
=
\bigl(
    \dim(V)\bmod 2,\,
    w_1^H(V),\,
    W_3^H(V)
\bigr),
\]
where
\[
w_1^H(V)\in H_H^1(\mathrm{pt};\mathbb{Z}/2)
\]
is the first equivariant Stiefel--Whitney class and
\[
W_3^H(V)
\coloneqq
\beta\bigl(w_2^H(V)\bigr)
\in H_H^3(\mathrm{pt};\mathbb{Z})
\]
is the third integral equivariant Stiefel--Whitney class. Here \(\beta\) denotes
the Bockstein homomorphism associated with the short exact sequence
\[
0\longrightarrow \mathbb{Z}
\xrightarrow{\times 2}
\mathbb{Z}
\longrightarrow \mathbb{Z}/2
\longrightarrow 0.
\]
We refer to \cite{DonovanKaroubi,FHTI} for the general theory of graded Brauer
groups and their relation to twisted \(K\)-theory.

\begin{theorem}[Twisted equivariant Thom isomorphism]
\label{thm:twisted-thom-representation}
Let \(H\) be a compact Lie group, and let \(V\) be a finite-dimensional
orthogonal real \(H\)-representation. Then there is a canonical isomorphism
\[
K_H^0(V)
\cong
K_H^{-\tau(V)}(\mathrm{pt}),
\]
where \(\tau(V)\in \operatorname{GBr}_H(\mathrm{pt})\) is the twisting
determined by the complex Clifford algebra \(\operatorname{Cl}(V)\).
\end{theorem}

This is precisely the case \(X=\mathrm{pt}\slash \slash H\), \(\sigma=0\), and
\(n=0\) of the general twisted Thom isomorphism
\cite[Equation~(3.19)]{FHTI}. 

In the terminology of \cite{FHTI}, for any twisting \( \tau \in \operatorname{GBr}_H(\mathrm{pt}) \), the group
\(K_H^{\tau}(\mathrm{pt})\) is called the \(\tau\)-\emph{twisted representation group} of \(H\) and is denoted by
\( R^{\tau}(H) \).

In the present case, the sign of the twisting is in fact immaterial. Indeed,
the representation \(V\oplus V\) carries the canonical \(H\)-invariant complex
structure
\[
J(v_1,v_2)=(-v_2,v_1),
\]
and hence admits an \(H\)-equivariant \(\operatorname{Spin}^c\)-structure.
Consequently, the associated Clifford twisting is Morita trivial, and therefore
\[
2\tau(V)
=
\tau(V\oplus V)
=
0
\qquad\text{in }\operatorname{GBr}_H(\mathrm{pt}).
\]
It follows that
\[
-\tau(V)=\tau(V).
\]
Thus, for notational convenience, we shall write the twisted Thom isomorphism
in the equivalent form
\[
K_H^0(V)
\cong
K_H^{\tau(V)}(\mathrm{pt}).
\]

In the untwisted case, equivariant \(K\)-theory of a point recovers the
complex representation ring of \(H\):
\[
K_H^0(\mathrm{pt})\cong R(H),
\qquad
K_H^1(\mathrm{pt})=0.
\]
Thus, elements of \(K_H^0(\mathrm{pt})\) are represented by virtual
finite-dimensional complex representations of \(H\). More generally, twisted
equivariant \(K\)-theory of a point may be interpreted in terms of twisted
representations. For the Clifford twisting \(\tau(V)\) considered above, this
interpretation can be made explicit using a canonical double cover of \(H\).

\begin{definition}
\label{def:pin-double-cover}
Suppose first that \(V\) is even-dimensional. Let
\[
q\colon \operatorname{Pin}^{-}(V)\longrightarrow O(V)
\]
denote the canonical double covering associated with the convention
\[
v^2=-\lVert v\rVert^2
\]
for the real Clifford algebra of \(V\). The
\emph{\(\operatorname{Pin}^{-}\)-double cover of \(H\) associated with
\(\rho\)} is the fiber product
\[
\widetilde H
\coloneqq
H\times_{O(V)}\operatorname{Pin}^{-}(V)
=
\left\{
(h,s)\in H\times\operatorname{Pin}^{-}(V)
\;\middle|\;
\rho(h)=q(s)
\right\}.
\]

If \(\rho\) is orientation preserving, so that
\[
\rho(H)\subseteq SO(V),
\]
then
\[
\widetilde H
=
H\times_{SO(V)}\operatorname{Spin}(V),
\]
and we call \(\widetilde H\) the \emph{Spin double cover of \(H\)}
associated with \(V\).

If \(V\) is odd-dimensional, we apply the preceding construction to
\(V\oplus\mathbb{R}\) and the representation
\[
\rho\oplus 1\colon H\longrightarrow O(V\oplus\mathbb{R}).
\]
\end{definition}

The group \(\widetilde H\) fits into a central extension
\[
1\longrightarrow \{\pm1\}
\longrightarrow \widetilde H
\overset{p}{\longrightarrow} H
\longrightarrow 1,
\]
where \(p(h,s)=h\). We denote the nontrivial element of the kernel by
\[
\epsilon\coloneqq(e,-1)\in\widetilde H.
\]

A finite-dimensional complex representation
\[
\pi\colon \widetilde H\longrightarrow U(W)
\]
is said to be \emph{genuine} if
\[
\pi(\epsilon)=-\operatorname{id}_W.
\]
We denote by
\[
R^{-}(\widetilde H)
\]
the Grothendieck group of finite-dimensional genuine complex representations
of \(\widetilde H\). Equivalently, \(R^{-}(\widetilde H)\) is the free
abelian group generated by the isomorphism classes of irreducible genuine
representations.

The group \(R^{-}(\widetilde H)\) carries a natural \(R(H)\)-module
structure. Indeed, if \(U\) is a finite-dimensional complex
\(H\)-representation and \(W\) is a genuine
\(\widetilde H\)-representation, then
\[
p^*U\otimes W
\]
is again genuine.

Combining the twisted equivariant Thom isomorphism with
\cite[Corollary~3.6]{echterhoff2009equivariant}, we obtain the following
description.

\begin{theorem}
\label{thm:twisted-k-theory-genuine-representations}
Let \(H\) be a compact Lie group, and let
\[
\rho\colon H\longrightarrow SO(V)
\]
be an orientation-preserving orthogonal representation on a finite-dimensional
real vector space \(V\). Let \(\widetilde H\) be the associated Spin double
cover. Then there is a natural isomorphism of \(R(H)\)-modules
\[
K_H^{\tau(V)}(\mathrm{pt})
\cong
\begin{cases}
R^{-}(\widetilde H),
    & \dim V \text{ is even},\\[4pt]
0,
    & \dim V \text{ is odd}.
\end{cases}
\]
\end{theorem}

Consequently, if the action of \(H\) on \(V\) is orientation preserving, then
the twisted Thom isomorphism and
Theorem~\ref{thm:twisted-k-theory-genuine-representations} yield
\[
K_H^{0}(V)
\cong
\begin{cases}
R^{-}(\widetilde H),
    & \dim V \text{ is even},\\[4pt]
0,
    & \dim V \text{ is odd}.
\end{cases}
\]
When the action is not orientation preserving, the corresponding description
is considerably more involved. In the next section, we restrict our attention
to the orientation-preserving case and give a more explicit description of the
classes \(\kappa_i\) appearing in Theorem~\ref{thm main}.

\section{The Equivariant Poincar\'e-Hopf theorem at  \(K\)-homology level}\label{subsec preimage computation}

In this section, we complete the proof of
Theorem~\ref{thm main} by establishing part~\textup{(3)}.
More precisely, we determine the local class associated with a
\(G\)-orientable component of the zero set. In this case, the isotropy
action on the corresponding normal fiber preserves orientation, although
it need not lift to a spin action.

We prove the following result.

\begin{theorem}\label{thm last}
Let \(Z_i=G/H_i\) be a \(G\)-orientable component of the zero set of
\(A\), and let \( V_i=(E_i)_{eH_i} \) be the normal fiber of \(Z_i\) at \(eH_i\). Then the class
\(\kappa_i\in R^{\tau_i}(H_i)\) appearing in
Theorem~\ref{thm main}\textup{(2)} can be chosen to be
\[
\kappa_i=
\begin{cases}
0,
& \dim V_i \text{ is odd},\\[4pt]
\deg\bigl(A|_{Z_i}\bigr)
\bigl([s_i^+]-[s_i^-]\bigr),
& \dim V_i \text{ is even},
\end{cases}
\]
where \(s_i^\pm\) are the half-spin representations of the double cover
\(\widetilde H_i\) introduced in
Definition~\ref{def:pin-double-cover}, and
\[
\deg\bigl(A|_{Z_i}\bigr) \in R^{\tau_{i}}(H_i)
\] 
will be defined below.
\end{theorem}

\begin{remark}
Since \(X\) is \(G\)-equivariantly spin, \(TX\) is
\(G\)-equivariantly oriented. Moreover, along \(Z_i\) we have the
\(G\)-equivariant decomposition
\[
TX|_{Z_i}\cong TZ_i\oplus E_i.
\]
Thus, if \(Z_i=G/H_i\) admits a \(G\)-equivariant orientation, then
the normal bundle \(E_i\) is also \(G\)-equivariantly oriented.
Consequently, the isotropy action of \(H_i\) on
\( V_i=(E_i)_{eH_i} \)
is orientation preserving.
\end{remark}

We begin by recalling the double cover
\(\widetilde H_i\to H_i\) introduced in
Definition~\ref{def:pin-double-cover} and establishing several properties
of its representations that will be used in the proof of
Theorem~\ref{thm last}.

Since
\[
K_{H_i}^0(V_i)=0
\]
when \(\dim V_i\) is odd, it remains only to consider the case in which
\(\dim V_i\) is even.

Let
\[
\rho_i\colon H_i\longrightarrow \mathrm{SO}(V_i)
\]
denote the isotropy representation on \(V_i\). The induced action of
\(\widetilde H_i\) on \(V_i\) is given by
\[
(h,s)\longmapsto \rho_i(h).
\]
By the definition of \(\widetilde H_i\), the homomorphism
\[
\widetilde{\rho}_i\colon \widetilde H_i\longrightarrow
\mathrm{Spin}(V_i),
\qquad
(h,s)\longmapsto s,
\]
lifts this orthogonal representation. Consequently, \(V_i\) carries a
natural \(\widetilde H_i\)-equivariant spin structure.

Let
\(i:\{0\}\longrightarrow V_i\)
denote the inclusion of the origin, and let
\(
\pi:V_i\longrightarrow\{0\}
\)
be the projection. Denote by \(\pi^*S_{V_i}^{\pm},\) the pullbacks to \(V_i\)
of the two half-spin representations. Define
\[
\mu:\pi^{*}S_{V_i}^{+}\longrightarrow\pi^{*}S_{V_i}^{-}
\]
fiberwise by
\[
\mu_v=c(v),\qquad v\in V_i,
\]
where \(c(v)\) denotes Clifford multiplication by \(v\).

Since
\(
c(v)^2=-\|v\|^2,
\)
the map \(\mu_v\) is invertible for \(v\neq 0\). Hence
\(
\widetilde{\beta}_{\widetilde H_i}
=
[\pi^{*}S_{V_i}^{+},\pi^{*}S_{V_i}^{-},\mu]
\in K_{\widetilde H_i}^{0}(V_i)
\)
is the \(\widetilde H_i\)-equivariant Bott class.Consequently:

\begin{proposition}
\label{thom}
The Thom isomorphism is realized by:
\begin{align}
\otimes\,\tilde{\beta}_{\tilde{H}_i} \cong i_{!} : \mathrm{R}(\tilde{H}_i) &\to K^{0}_{\tilde{H}_i}(V_{i}) \\
 a &\mapsto a \cdot \tilde{\beta}_{\tilde{H}_i}
\end{align}
where $\tilde{\beta}_{\tilde{H}_i} = [S_{V_{i}}^{+}, S_{V_{i}}^{-}, \mu] \in K_{\tilde{H}_i}^{0}(V_{i})$ denotes the $\tilde{H}_i$-equivariant Bott element.
\end{proposition}

Let $(1, \epsilon)$ be the central element of $\tilde{H}_i$. This induces a decomposition  
$K_{\tilde{H}_i}^{0}(V_{i}) = K_{\tilde{H}_i}^{0}(V_{i})^{+} \oplus K_{\tilde{H}_i}^{0}(V_{i})^{-}$,  
where the superscript $\pm$ indicates the eigenspaces of the $(1,\epsilon)$-action ($\pm I$ respectively) \label{negitive}. The map  
\[i_{!}: \mathrm{R}(\tilde{H}_i)^{-} \to K_{\tilde{H}_i}^{0}(V_{i})^{+} \cong K_{H_i}^{0}(V_{i})\]
is an isomorphism. Combining this with Proposition \ref{thom} yields:

\begin{theorem}\label{h}
The following map is an isomorphism:
\begin{align}
\otimes\,\tilde{\beta}_{\tilde{H}_i} \cong i_{!} : \mathrm{R}(\tilde{H}_i)^* &\to K^{0}_{H_i}(V_{i}) \\
 a &\mapsto a \cdot \tilde{\beta}_{\tilde{H}_i}
\end{align}
where $\tilde{\beta}_{\tilde{H}_i}$ is the Bott element defined above, and  
\[
\mathrm{R}(\tilde{H}_i)^{*} = \begin{cases}
\mathrm{R}(\tilde{H}_i)^{+} & \text{if } V_{i} \text{ admits an equivariant spin structure}, \\
\mathrm{R}(\tilde{H}_i)^{-} & \text{otherwise}.
\end{cases}
\]
\end{theorem}

Then, to prove the Theorem \ref{thm last}, one needs to compute the preimage of $[{A}_{E_i}]$ under $\otimes\tilde{\beta}_{\tilde{H_i}}.$

 Let $S_{Z_i}$ be the spinor bundle on $ Z_i.$ Recall that $D^{Z_i}$ is the Dirac-type operator obtained by twisting  $\D^{Z_i}$ with $S_{Z_i}.$  According to the definition, $Z_i$ possesses a G-invariant tubular neighborhood $(h, E_i)$. In this case,
 $E_i$ splits into two orthogonal G-subbundles $E_i=E_i^{+}\oplus E_i^{-}$. Moreover, $h: E_i\rightarrow X$ is an equivariant embedding and there is an open G-invariant neighborhood $U_i$ of $Z_i$ in $E_i$
such that for any $e=(e_{+},e_{-})\in E_i\cap U_i,$ we have $f(h(e))=c-\mid e_{-}\mid ^{2}+\mid e_{+}\mid^{2},$ where $f$ is the $G$-invariant Morse-Bott function, and $c$ denotes the value of the  constant function $f|_{Z_i}$. 

Let \(r=\operatorname{rank}(E_i^-)\), and let
\(\operatorname{Fr}_{O}(E_i^-)\) denote the orthonormal frame bundle of
\(E_i^-\). We define the orientation line bundle of \(E_i^-\) by
\[
O(E_i^-)
:=
\operatorname{Fr}_{O}(E_i^-)\times_{\det}\mathbb{R},
\]
where
\[
\det:O(r)\longrightarrow\{\pm1\}
\]
acts on \(\mathbb{R}\) by scalar multiplication. Thus,
\(O(E_i^-)\) is a one-dimensional real line bundle over \(Z_i\).
Since \(E_i^-\) is \(G\)-equivariant, \(O(E_i^-)\) naturally inherits
a \(G\)-equivariant structure.

The following proposition explains the relation between $[{A}_{E_i}|_{\mathbb{R}^{n}}]\in K_{H_{i}}^{0}(\mathbb{R}^{n_{i}})$ and 
\[[S^{+}(T^{*}(E_i))|_{\mathbb{R}^{n_i}},S^{-}(T^{*}(E_i))|_{\mathbb{R}^{n_i}},c(v)|_{\mathbb{R}^{n_i}}]\in K_{H_{i}}^{0}(\mathbb{R}^{n_i}).\]  
\begin{proposition}\label{sg}
    There exists $\mathrm{deg}(A)|_{Z_{i}}\in \mathrm{R}(H_{i})\subset R^{\tau_{i}}(H_{i})$ such that:
    $$[{A}_{E_i}|_{\mathbb{R}^{n_{i}}}]=(\mathrm{deg}(A)|_{Z_{i}})\cdot[S^{+}(T^{*}(E_i))|_{\mathbb{R}^{n_{i}}},S^{-}(T^{*}(E_i))|_{\mathbb{R}^{n_{i}}},c(v)|_{\mathbb{R}^{n_{i}}}].$$ 
\end{proposition}

\begin{proof}

First, we consider the case where $\mathrm{dim}(E_i^{-})$ is odd. In this case, it is clear that $O(E_i^{-}\otimes O(E_i^{-}))$ is a trivial bundle since $E_i^{-}\otimes O(E_i^{-})$ is orientable. Thus, the following bundle isomorphism
    \begin{align*}
        F:\Gamma(O(E_i^{-}\otimes O(E_i^{-})))&\rightarrow \Gamma(O(E_i^{-}\otimes O(E_i^{-})))\\
        & s \mapsto (-1)^{\mathrm{dim}(E_i^{-})}s
        \end{align*}
        induces another equivariant bundle isomorphism
        \begin{align*}
            \phi_{i}: E_i^{-}\otimes O(E_i^{-})&\rightarrow E_i^{-}\otimes O(E_i^{-})\\
            v_{z}&\mapsto  -v_{z}.
        \end{align*}
         
         Therefore, there exists an equivariant isomorphism $\hat{\phi_{i}}=\mathrm{Id}_{E_i^{+}\otimes O(E_i^{-})}\oplus\phi_{i}:E_i\otimes O(E_i^{-})\rightarrow E_i\otimes O(E_i^{-})$. Set $\sigma\in \Gamma(O(E_i^{-}))$ and $\sigma^{*}\in \Gamma(O(E_i^{-})^{*}))$ such that $\sigma\otimes \sigma^{*}\in \Gamma(O(E_i^{-})\otimes O(E_i^{-})^{*})$ is a non-vanishing section. Hence
         \[ E_i\cong E_i\otimes O(E_i^{-})\otimes(O(E_i^{-}))^{*}.\]
         Thus, we have an equivariant bundle isomorphism between $E_i$ factored through:
         \begin{align*}
             \hat{\phi_{i}}\otimes\mathrm{Id}:  E_i\otimes O(E_i^{-})\otimes(O(E_i^{-}))^{*}&\rightarrow  E_i\otimes O(E_i^{-})\otimes(O(E_i^{-}))^{*}\\
             v_{z}\otimes \sigma\otimes\sigma^{*}&\mapsto v^+_z\otimes \sigma\otimes\sigma^{*}+(-v^-_{z}\otimes \sigma)\otimes\sigma^{*}.
             \end{align*}

        Therefore, we know that $(\hat{\phi_{i}}\otimes\mathrm{Id})|_{\mathbb{R}^{n_i}}$ induces an $H_{i}$-equivariant isomorphism from $\mathbb{R}^{n}$ to itself.  This equivariant isomorphism induces an $\mathrm{R}(H_{i})$-module isomorphism at the level of K-theory:$$((\hat{\phi_{i}}\otimes\mathrm{Id})|_{\mathbb{R}^{n_i}})^*:K_{H_{i}}^{0}(\mathbb{R}^{n_i})\rightarrow K_{H_{i}}^{0}(\mathbb{R}^{n_i}).$$  

        Since the orthogonal decomposition
\[
E_i=E_i^{+}\oplus E_i^{-}
\]
is \(G\)-invariant, it remains \(H_i\)-invariant after restricting the
\(G\)-action to \(H_i\). Hence,
\[
(\hat{\phi_i}\otimes\mathrm{Id})|_{\mathbb{R}^{n_i}}
\]
is \(H_i\)-equivariant and, under the above decomposition, acts on the
fiber \(\mathbb{R}^{n_i}\) by
\[
(v^{+},v^{-})\longmapsto(v^{+},-v^{-}).
\]
Moreover, the pullback induces grading-preserving \(H_i\)-equivariant
bundle isomorphisms
\[
\left(
(\hat{\phi_i}\otimes\mathrm{Id})|_{\mathbb{R}^{n_i}}
\right)^{*}
\left(
S^{\pm}(T^{*}(E_i))|_{\mathbb{R}^{n_i}}
\right)
\cong
S^{\pm}(T^{*}(E_i))|_{\mathbb{R}^{n_i}}.
\]
Through these equivariant bundle isomorphisms, the pulled-back symbol is
computed as follows. For \(v=(v^{+},v^{-})\), we have
\[
\begin{aligned}
&\left(
(\hat{\phi_i}\otimes\mathrm{Id})|_{\mathbb{R}^{n_i}}
\right)^{*}(A_{E_i})(v^{+},v^{-})
\\
&\qquad =
A_{E_i}(v^{+},-v^{-})
\\
&\qquad =
c(v^{+})-c(-v^{-})
\\
&\qquad =
c(v^{+})+c(v^{-})
\\
&\qquad =
c(v).
\end{aligned}
\]
Here we have used the linearity of Clifford multiplication,
\[
c(-v^{-})=-c(v^{-}).
\]
Therefore,
\[
\begin{aligned}
&\left(
(\hat{\phi_i}\otimes\mathrm{Id})|_{\mathbb{R}^{n_i}}
\right)^{*}
\left([A_{E_i}|_{\mathbb{R}^{n_i}}]\right)
\\
&\qquad =
\left[
S^{+}(T^{*}(E_i))|_{\mathbb{R}^{n_i}},
S^{-}(T^{*}(E_i))|_{\mathbb{R}^{n_i}},
c(v)|_{\mathbb{R}^{n_i}}
\right].
\end{aligned}
\]
Thus, flipping the negative subspace \(E_i^{-}\) transforms the
Clifford symbol \(c(v^{+})-c(v^{-})\) into the standard Clifford
symbol \(c(v^{+})+c(v^{-})=c(v)\).

        Since $$((\hat{\phi_{i}}\otimes\mathrm{Id})|_{\mathbb{R}^{n_i}})^*:K_{\tilde{H_{i}}}^{0}(\mathbb{R}^{n_i})\rightarrow K_{\tilde{H_{i}}}^{0}(\mathbb{R}^{n_i})$$
        is also an $\mathrm{R}(\tilde{H_{i}})$-module isomorphism and preserves the grading of
       \[K_{\tilde{H_{i}}}^{0}(\mathbb{R}^{n_i})=K_{\tilde{H_{i}}}^{0}(\mathbb{R}^{n_i})^{+}\oplus K_{\tilde{H_{i}}}^{0}(\mathbb{R}^{n_i})^{-},\]
       where $K_{\tilde{H_{i}}}^{0}(\mathbb{R}^{n_i})^{+}\cong K_{H_{i}}^{0}(\mathbb{R}^{n_i})$.
Through this bundle isomorphism, we conclude that there exists $\mathrm{deg}(A)|_{Z_{i}}\in \mathrm{R}(H_{i})\cong \mathrm{R}(\tilde{H_{i}})^+\subset R^{\tau_{i}}(H_{i})$  where \[\mathrm{deg}(A)|_{Z_{i}}\otimes \mathrm{deg}(A)|_{Z_{i}}=\mathrm{Id}_{H_{i}}^{\mathbb{C}}\in \mathrm{R}(H_{i}),\] 
here $\mathrm{Id}_{H_{i}}^{\mathbb{C}}$ is the one-dimensional complex trivial representation of $H_{i}$ such that for any $x\in K_{H_{i}}^{0}(\mathbb{R}^{n_i}),$ we have:
$$((\hat{\phi_{i}}\otimes\mathrm{Id})|_{\mathbb{R}^{n_i}})^*(x)=\mathrm{deg}(A)|_{Z_{i}}\cdot x,$$
here $A=d f$ is the $G$-invariant Morse–Bott differential 1-form over $X.$
Thus, the following two K-theory elements are equal:
         $$[A_{E_i}|_{\mathbb{R}^{n_{i}}}]=(\mathrm{deg}(A)|_{Z_{i}})\cdot[S^{+}(T^{*}(E_i))|_{\mathbb{R}^{n_{i}}},S^{-}(T^{*}(E_i))|_{\mathbb{R}^{n_{i}}},c(v)|_{\mathbb{R}^{n_{i}}}].$$ 
         
         Therefore, we have proved our claim in the odd-dimensional case.

        Then, we consider the case when $\mathrm{dim}({E_i}^{-})$ is even. In this case, there exists a natural bundle isomorphism: \begin{align*}
            \phi:  &E_i^{-}\rightarrow E_i^{-}\\
            &(x,v)\mapsto (x,-v)
        \end{align*}
   since $\mathrm{dim}(E_{i}^{-})$ is even. Similar to the odd-dimensional case, through the bundle isomorphism $\mathrm{Id}\oplus\phi: E_i\cong E_i^{+}\oplus E_i^{-}\rightarrow E_i\cong E_i^{+}\oplus E_i^{-}$, we identify the following two $K$-theory elements: 
        $$[A_{E_i}|_{\mathbb{R}^{n_{i}}}]=(\mathrm{deg}(A)|_{Z_{i}})\cdot[S^{+}(T^{*}(E_i))|_{\mathbb{R}^{n_{i}}},S^{-}(T^{*}(E_i))|_{\mathbb{R}^{n_{i}}},c(v)|_{\mathbb{R}^{n_{i}}}].$$

       Therefore, we complete the proof.   
\end{proof}

 \begin{theorem}
 \label{main}
     The class  $[A_{E_{i}}|_{\mathbb{R}^{n_i}}]\in K^{0}_{H_{i}}(\mathbb{R}^{n_{i}})$ equals 
    {\small $$(\mathrm{deg}(A)|_{Z_{i}})([\pi^{*}(S^+(T^{*}_{[e]}(G/H_{i})))]-[\pi^{*}(S^-(T^{*}_{[e]}(G/H_{i})))])\cdot[S^+(T^*(\mathbb{R}^{n_{i}})),S^-(T^*(\mathbb{R}^{n_{i}})),c(x)],$$ }
     where $\pi:\mathbb{R}^{n_{i}}\cong H_{i}\times_{H_{i}}\mathbb{R}^{n_{i}}\rightarrow \{0\}$, $S^{\pm}(T^{*}_{[e]}(G/H_{i}))$ is the spinor space of $T^{*}_{[e]}(G/H_{i})$ and $[S^+(T^*(\mathbb{R}^{n_{i}})),S^-(T^*(\mathbb{R}^{n_{i}})),c(x)]\in K^0_{\tilde{H_{i}}}(\mathbb{R}^{n_{i}})^{-}$ is the $\tilde{H_{i}}$-equivariant Bott element.
 \end{theorem}
 \begin{proof}
     Since $T^{*}(E_{i})$ has a $G$-equivariant Spin structure, it is also $H_{i}$-equivariant. This $H_{i}$-equivariant Spin structure naturally induces an $\tilde{H_{i}}$-equivariant Spin structure, which makes $T^{*}(E_{i})|_{\mathbb{R}^{n_{i}}}$ an $\tilde{H_{i}}$-Spin bundle.

     Let $\pi: \mathbb{R}^{n_{i}} \rightarrow \{0\}$. We know that $T^{*}(E_{i})|_{\mathbb{R}^{n_{i}}} \cong \pi^{*}(T^{*}_{[e]}(G/H_{i})) \oplus T^{*}(\mathbb{R}^{n_{i}})$ and this decomposition is $\tilde{H_{i}}$-equivariant. Furthermore, since by assumption $T^{*}(\mathbb{R}^{n_{i}})$ has an $\tilde{H_{i}}$-equivariant Spin structure, it follows that $\pi^{*}(T^{*}_{[e]}(G/H_{i}))$ also has an $\tilde{H_{i}}$-equivariant Spin structure. Therefore, $\pi^{*}(S(T^{*}_{[e]}(G/H_{i})))$ is an $\tilde{H_{i}}$-Spin bundle. 
     
     Since both  $\pi^{*}(S(T^{*}_{[e]}(G/H_{i})))$ and $S(T^*(\mathbb{R}^{n_{i}}))$ are $\tilde{H_{i}}$-Spin bundles and \\
     $T^{*}(E_{i})|_{\mathbb{R}^{n_{i}}} \cong \pi^{*}(T^{*}_{[e]}(G/H_{i})) \oplus T^{*}(\mathbb{R}^{n_{i}})$ (this decomposition is $\tilde{H_{i}}$-equivariant), we have  $S(T^*(E_{i}))|_{\mathbb{R}^{n_{i}}}\cong \pi^*(S(T^*_{[e]}(G/H_{i})))\otimes S(T^*(\mathbb{R}^{n_{i}})).$ Since this isomorphism is $\tilde{H_{i}}$-equivariant, we have:
     
     {\small \begin{eqnarray*}
         &&[A_{E_{i}}|_{\mathrm{R}^{n_i}}]\\
         &=&(\mathrm{deg}(A)|_{Z_{i}})[S^+(T^*(E_{i}))|_{\mathbb{R}^{n_{i}}},S^-(T^*(E_{i}))|_{\mathbb{R}^{n_{i}}},c(v)|_{\mathbb{R}^{n_{i}}}]\\
         &=&(\mathrm{deg}(A)|_{Z_{i}})[\pi^*(S^+(T^*_{[e]}(G/H_{i})))\otimes S^{+}(T^*(\mathbb{R}^{n_{i}}))\oplus\pi^*(S^-(T^*_{[e]}(G/H_{i})))\otimes S^-(T^*(\mathbb{R}^{n_{i}})),\\
         &&\pi^*(S^+(T^*_{[e]}(G/H_{i})))\otimes S^{-}(T^*(\mathbb{R}^{n_{i}}))\oplus\pi^*(S^-(T^*_{[e]}(G/H_{i})))\otimes S^+(T^*(\mathbb{R}^{n_{i}})),\mathrm{Id}\otimes c(x)],
     \end{eqnarray*} }
     which equals 
   {\small  \[ (\mathrm{deg}(A)|_{Z_{i}})([\pi^{*}(S^+(T^{*}_{[e]}(G/H_{i})))]-[\pi^{*}(S^-(T^{*}_{[e]}(G/H_{i})))])\cdot[S^+(T^*(\mathbb{R}^{n_{i}})),S^-(T^*(\mathbb{R}^{n_{i}})),c(x)].\]}
     The proof is then completed.\end{proof}
     
 Thus we have:
 \begin{proposition}
 \label{f}
     Denote the elements in $\mathrm{R}(\tilde{H_{i}})\subset R^{\tau_{i}}(H_{i})$ derived from the action of $\tilde{H_{i}}$ on $S^{\pm}(T^{*}_{[e]}(G/H_i))$ as $[s_{i}^{\pm}]$. Following from Proposition~\ref{thom} and Theorem~\ref{main} we have:
     $$((\mathrm{deg}(A)|_{Z_{i}})\cdot([s_{i}^+]-[s_{i}^-]))\otimes \tilde{\beta_{i}}=[A_{E_{i}}|_{\mathbb{R}^{n_i}}]\in K^{0}_{H_{i}}(\mathbb{R}^{n_{i}}).$$
 \end{proposition}

Finally, we shall prove Theorem~\ref{thm main}:

\begin{proof}[Proof of Proposition~\ref{thm main}~(3)]
By Proposition~\ref{f},
\[
\left(
\deg(A)|_{Z_i}
\right)
\bigl([s_i^+]-[s_i^-]\bigr)
\otimes \tilde{\beta}_i
=
[A_{E_i}|_{\mathbb{R}^{n_i}}]
\in K_{H_i}^0(\mathbb{R}^{n_i}).
\]
Therefore,
\begin{align*}
\left.T\right|_{R^{\tau_i}(H_i)}
\left(
\left(\deg(A)|_{Z_i}\right)
\bigl([s_i^+]-[s_i^-]\bigr)
\right)
&=
\iota_*\circ\mathrm{PD}\circ\mathrm{Inc}
\left(
[A_{E_i}|_{\mathbb{R}^{n_i}}]
\right)
\\
&=
\iota_*\circ\mathrm{PD}
\left(
[A_{E_i}]
\right).
\end{align*}
The first equality follows from Proposition~\ref{f}.
\end{proof}

\section{Declaration of generative AI and AI-assisted technologies in the writing process}

During the preparation of this work the authors used ChatGPT in order to improve the readability and language of the manuscript. After using this tool, the authors reviewed and edited the content as needed and take full responsibility for the content of the published article.

\bibliography{main}{}
\bibliographystyle{amsplain} 

\end{document}